\documentclass[a4paper,twoside,10pt]{article}

\usepackage[english]{babel}
\usepackage[utf8]{inputenc}
\usepackage{lmodern,color}
\usepackage[T1]{fontenc}
\usepackage{indentfirst}
\usepackage{amsmath,amssymb}
\usepackage[percent]{overpic}
\usepackage[a4paper,top=3cm,bottom=3cm,left=3cm,right=3cm,%
bindingoffset=5mm]{geometry}
\usepackage{lmodern}
\usepackage[T1]{fontenc}
\usepackage{lettrine}
\usepackage{booktabs}
\usepackage{subcaption}
\usepackage{authblk}
\usepackage{emp}
\usepackage{amsthm}
\usepackage{amsmath,amssymb,amsthm}
\usepackage{braket}
\usepackage{chngpage}
\usepackage{cases}
\usepackage{graphicx}
\usepackage{epstopdf}
\usepackage{tabularx}
\usepackage{multirow}
\usepackage{enumerate}
\usepackage{todonotes}
\usepackage{mathtools}

\newcommand{\numberset}{\mathbb}
\newcommand{\N}{\numberset{N}}
\newcommand{\R}{\numberset{R}}

\newcommand{\Pk}{\numberset{P}}

\renewcommand{\epsilon}{\varepsilon}
\renewcommand{\theta}{\vartheta}
\renewcommand{\rho}{\varrho}
\renewcommand{\phi}{\varphi}

\newcommand{\xx}{\boldsymbol{x}}

\newcommand{\nn}{\boldsymbol{n}}

\newcommand{\bb}{\boldsymbol{\beta}}

\newcommand{\dd}{{\rm div}}
\newcommand{\gr}{\nabla}

\newcommand{\vek}{\mathcal{V}^{E,k}}
\newcommand{\eek}{\mathcal{E}^{E,k}}
\newcommand{\pek}{\mathcal{P}^{E,k}}
\newcommand{\vk}{\mathcal{V}^{k}}
\newcommand{\ek}{\mathcal{E}^{k}}
\newcommand{\pk}{\mathcal{P}^{k}}

\def\tri{{| \! | \! |}}

\def\PN{{\Pi^{\nabla, E}_{k}}}
\def\P0{{\Pi^{0, E}_{k}}}
\def\Pg{{\Pi^{0}_{k}}}
\def\PZ0P{{\boldsymbol{\Pi}^{0, E}_{k-1}}}
\def\PP0P{{\boldsymbol{\Pi}^{0, E}_{k}}}
\def\Stab{\mathcal{S}^E}
\def\AcipE{\mathcal{A}_{\rm cip}^E}

\def\Acip{\mathcal{A}_{\rm cip}}

\def\AtcipE{\widetilde{\mathcal{A}}_{\rm cip}^E}

\def\Atcip{\widetilde{\mathcal{A}}_{\rm cip}}

\def\vint{v_{\mathcal{I}}}
\def\uint{u_{\mathcal{I}}}
\def\eint{e_{\mathcal{I}}}

\def\epi{e_{\pi}}

\def\bskew{{b^{\rm skew}}}
\def\bskewE{{b^{{\rm skew},E}}}
\def\bskewh{{b^{{\rm skew}}_h}}
\def\bskewEh{{b^{{\rm skew},E}_h}}

\def\bbO{{\|\bb\|_{[L^{\infty}(\Omega)]^2}}}

\def\cip{{\rm cip}}
\def\cipdual{{\rm cip^*}}
\def\errF{\eta_{\mathcal{F}}^E}

\def\erra{\eta_{a}^E}
\def\pwh{  h\pi ( \bb_h \cdot \nabla \Pg v_h )}
\def\wh{  h \bb_h \cdot \nabla \Pg v_h}

\def\errb{\eta_{b}^E}
\def\errc{\eta_{c}^E}
\def\errJ{\eta_{J}^E}
\def\errN{\eta_{N}^E}

\def\errbA{\eta_{b,A}^E}
\def\errbB{\eta_{b,B}^E}

\def\eNa{\eta_{N,a}^E}
\def\eNb{\eta_{N,b}^E}
\def\eNc{\eta_{N,c}^E}

\def\reg{s}

{\left\lbrace\begin{array}{@{}l@{}}}%
{\end{array}\right.}
\theoremstyle{definition}

\theoremstyle{remark}
\newtheorem{remark}{Remark}[section]

\theoremstyle{remark}

\theoremstyle{plain}

\newtheorem{proposition}{Proposition}[section]

\newtheorem{lemma}{Lemma}[section]

\usepackage{lipsum}

\usepackage{authblk}
\usepackage{color}

\usepackage{setspace}
\newcommand{\Manuel}[1]{\noindent{\color{blue}\textbf{[M:~#1]}}}

\author[1,2]{L. Beir\~ao da Veiga \thanks{lourenco.beirao@unimib.it}}

\author[3]{C. Lovadina \thanks{carlo.lovadina@unimi.it}}

\author[4]{M. Trezzi \thanks{manuelluigi.trezzi01@universitadipavia.it}}

\affil[1]{Dipartimento di Matematica e Applicazioni,
Universit\`a degli Studi di Milano Bicocca,
Via Roberto Cozzi 55 - 20125 Milano, Italy}  

\affil[2]{IMATI-CNR, Via Adolfo Ferrata 5 - 27100 Pavia, Italy}

\affil[3]{Dipartimento di Matematica ``F. Enriques'',
Universit\`a degli Studi di Milano,
Via Cesare Saldini 50 - 20133 Milano, Italy}

\affil[4]{Dipartimento di Matematica ``F. Casorati'',
Universit\`a di Pavia,
Via Adolfo Ferrata 5 - 27100 Pavia, Italy}

\title{\textbf{CIP-stabilized Virtual Elements for diffusion-convection-reaction problems}}

\date{\today}

\begin{document}

\maketitle

\begin{abstract}
The Virtual Element Method for diffusion-convection-reaction problems is considered. In order to design a quasi-robust scheme also in the convection-dominated regime, a Continuous Interior Penalty approach is employed. Due to the presence of polynomial projection operators, typical of the Virtual Element Method, the stability and the error analysis requires particular care, especially in treating the advective term. Some numerical tests are presented to support the theoretical results.  
\end{abstract}

\section{Introduction}
\label{sec:intro}

The Virtual Element Method (VEM) is a fairly recent methodology for the discretization of problems in partial differential equations \cite{volley, hitchhikers}, which can be interpreted as a generalization of classical Finite Elements (FEM) to meshes of much more general shape. Since its birth, the VEM has enjoyed a large success and been applied to a very wide range of problems; we here limit ourselves in mentioning the recent special issue \cite{sema-simai} and the review paper \cite{ACTA-VEM}.

The focus of the present article is on the classical diffusion-reaction-advection scalar problem. 
Under suitable assumptions on the data, this is a standard ``textbook'' elliptic problem without any particular difficulty. On the other hand, it is well known that, whenever the advective term dominates (in particular over the diffusive one) a classical FEM approach will lead to very large errors and oscillations in the discrete solution, unless an extremely fine mesh is adopted. There is a large FEM literature on the subject, offering a list of possible stabilized methods which are robust in this respect. From the theoretical standpoint, a method is typically called quasi-robust if, assuming sufficiently regular solution and data, it yields error estimates which are uniform with respect to the diffusion parameter in a norm including also some direct control on the convective term. Some well known approaches are upwind Discontinuous Galerkin schemes \cite{DE-book,R-book,BrezziDG}, Streamline Upwind Petrov-Galerkin and variants \cite{brooks1982}, Continuous Interior Penalty (CIP) \cite{douglas, burman:2004, burman:2007}, Local Projection Stabilization \cite{LPS0,matthies2007unified}. Finally, one must note that the diffusion-reaction-advection problem serves also as a model for more complex problems in fluid mechanics, such as the Navier-Stokes equation.

The Virtual Element Method is particularly suitable in the context of advection dominated problems due to the flexibility of the mesh construction and its handling. For instance, VEM allows more local refinement procedures and easy an gluing of fine meshes with coarser ones (this latter feature is very useful in the presence of layers, for example). In addition, VEM offers a more efficient discretization of complex domains, which is greatly useful in applications such as reservoir \cite{andersen17} and fracture-network simulations \cite{gen_Benedetto2014a}, where diffusion-reaction-advection equations play a crucial role. Unfortunately, due to the presence of projection operators which may alter the structure of the convective term, it is not easy to devise and analyze quasi-robust VEM schemes. Exceptions are the SUPG and LPS approaches detailed in \cite{berrone:2016,BDLV:2021} and \cite{li2021}, respectively (regarding other polygonal technologies, see for instance \cite{HHO-book-1,HHO-book-2}). 


Since three of the most popular stabilization techniques, namely SUPG, LPS and CIP, have their own strongly defined set of assets/drawbacks, broadening the available approaches with CIP schemes is important for the VEM technology. 
The purpose of the present contribution is exactly to fill this gap and develop CIP (Continuous Interior Penalty) stabilized VEM method, providing also a theoretical error analysis. Of course, our method combines VEM stabilization terms (to deal with polygonal meshes) and CIP-like terms (to deal with the avdection-dominated regime). 
Furthermore, it is worth noticing that the backstage complex nature of CIP, which is a ``minimal stabilization'' as it adds the minimal positive term guaranteeing control on piecewise polynomial convection, makes the analysis in the VEM setting particularly interesting and challenging. Assuming, as it happens in most publications on the subject, a uniformly positive reaction term, we are able to develop quasi-robust error estimates for our method. In the absence of reaction, we are able to show some improved error estimates (over a non-stabilized scheme), but only under a piecewise polynomial convection data assumption. The paper ends with a set of numerical tests showing the actual robustness of the method and comparing it with the non-stabilized approach.

The paper is organized as follows. After presenting the continuous and discrete problems in Section \ref{sub:VEM}, we develop the stability and convergence analysis in Section \ref{sec:theory}. Finally, numerical tests are shown in Section \ref{sec:num}.

Throughout the paper, we use standard notations for Sobolev norms and semi-norms. Moreover, $C$ and $C_i$ will denote quantities, independent of the meshsize $h$, which may vary at each occurrence. We will make extensive use of the notation $a \lesssim b$ ($a$ and $b$ being non-negative quantities) to mean $a\leq C b$.   

\section{The continuous and the discrete problems}
\label{sub:VEM}

In this Section we deal with the continuous problem and its discretization by means of the Virtual Element Method.

\subsection{Continuous Problem}
\label{sec:cp}


We consider the following steady advection-diffusion-reaction problem: 
\begin{equation}\label{eq:problem-s}
\left\{
\begin{aligned} 
&\mbox{find $u:\Omega \to \R$ such that:}\\
 &- \epsilon \Delta u + \bb \cdot \nabla u + \sigma u 
 = f \qquad \text{in }\Omega , \\
 &u_{|\Gamma} = 0 ,
\end{aligned}
\right.
\end{equation}
where $\Omega \subset \R^2$ is a polygonal domain of boundary $\Gamma$. Above, $\epsilon > 0$ is the diffusion coefficient (assumed to be constant), while 
$\bb \in [W^{1,\infty}(\Omega)]^2$ is the advection field such that $\dd \bb = 0$. Moreover, 
$\sigma>0$ is the reaction constant (except for Section \ref{ss:p1beta}, where $\sigma=0$); we remark that we assume $\sigma$ to be a positive constant since the extension to the case $0<\sigma\in L^\infty(\Omega)$, with $\sigma^{-1}\in L^\infty(\Omega)$, is trivial.
Finally, $f \in L^2(\Omega)$ is the source term.

The domain boundary will be split into two non-overlapping regions:
\begin{equation*}
\Gamma_{\text{in}} \coloneqq \{ \xx \in \Gamma \, | \, (\bb(\xx) \cdot \nn) < 0 \}
\quad \text{and} \quad 
\Gamma_{\text{out}} \coloneqq \{ \xx \in \Gamma \, | \, (\bb(\xx) \cdot \nn) \geq 0 \} \, ,
\end{equation*}
where $\nn$ is the outward unit normal vector to the boundary. 

A variational formulation of problem \eqref{eq:problem-s} reads as follows:
\begin{equation}\label{eq:problem-c}
\left\{
\begin{aligned} 
&\text{find $u \in V \coloneqq H^1_0(\Omega)$ such that: } \\
&\epsilon \, a(u,v) + \bskew(u,v) + \sigma \, c(u,v)  = \int_\Omega f \, v \, {\rm d}\Omega \, .
\end{aligned}
\right.
\end{equation}
The bilinear forms
$a(\cdot,  \cdot) \colon V \times V \to \R$ , 
$\bskew(\cdot,  \cdot) \colon V \times V \to \R$
and
$c(\cdot,  \cdot) \colon V \times V \to \R$
are defined as
\begin{equation}
\label{eq:a-c}
a(u,  v) \coloneqq \int_{\Omega} \nabla u \cdot \nabla v \, {\rm d}\Omega
\qquad \text{for all $u, v \in V$,}
\end{equation}
\begin{equation}
\label{eq:b-c}
\bskew (u,v) \coloneqq \dfrac{1}{2}\bigl(b(u,  v) - b(v, u)\bigr) \quad\text{with}\quad b(u,  v) := \int_{\Omega} (\bb \cdot \nabla u) \, v \, {\rm d}\Omega 
\qquad \text{for all $u, v \in V$,}
\end{equation}
\begin{equation}
\label{eq:c-c}
c(u,  v) := \int_{\Omega}  u \, v \, {\rm d}\Omega
\qquad \text{for all $u, v \in V$.}
\end{equation} 
%

%
It is well known that when $\epsilon$ is small with respect to $\bb$ and/or to $\sigma$, standard discretizations of \eqref{eq:problem-c} typically return unsatisfactory numerical solutions with spurious oscillations. 
To overcome these difficulties, several strategies are available in the literature.

In this paper we take advantage of the so-called Continuous Interior Penalty (CIP) strategy, introduced in 
\cite{burman:2004} in a Finite Element framework.
From now on, we assume that the material parameters are scaled so that we have 
\begin{equation}
\label{eq:beta-scaling}
\bbO=1\, .
\end{equation}
%

%
%
%
%
%

\subsection{Preliminary notations and results}
\label{sub:proj}
We start considering a sequence $\set{\Omega_h}_h$ of tessellations of $\Omega$ into non-overlapping polygons $E$.  
We denote with $e$ a general edge of $E$, while 
$|E|$ and $h_E$ are the area and the diameter of $E$, respectively. 
Furthermore, $\nn^E$ is the unit outward normal vector to the boundary $\partial E$. 
As usual, $h \coloneqq \sup_{E\in\Omega_h}h_E$ denotes the mesh parameter.
We suppose that $\set{\Omega_h}_h$ fulfils the following assumption:\\
\textbf{(A1) Mesh assumption.}
There exists a positive constant $\rho$ such that for any $E \in \set{\Omega_h}_h$   
\begin{itemize}
\item $E$ is star-shaped with respect to a ball $B_E$ of radius $ \geq\, \rho \, h_E$;
\item any edge $e$ of $E$ has length  $ \geq\, \rho \, h_E$;
\item the mesh is quasi-uniform, any polygon has diameter $h_E \geq \rho h$.
\end{itemize} 

We now introduce some basic tools and notations useful in the construction and the theoretical analysis of Virtual Element Methods.

Using standard VEM notations, for $n \in \N$, $m\in \N$ and $p\in [1, +\infty]$, and for any $E \in \Omega_h$,  
let us introduce the spaces:
\begin{itemize}
\item $\Pk_n(E)$: the set of polynomials on $E$ of degree $\leq n$  (with $\Pk_{-1}(E)=\{ 0 \}$),
\item $\Pk_n(\Omega_h) := \{q \in L^2(\Omega) \quad \text{s.t} \quad q|_E \in  \Pk_n(E) \quad \text{for all $E \in \Omega_h$}\}$,
\item $W^m_p(\Omega_h) := \{v \in L^2(\Omega) \quad \text{s.t} \quad v|_E \in  W^m_p(E) \quad \text{for all $E \in \Omega_h$}\}$ equipped with the broken norm and seminorm
\[
\begin{aligned}
&\|v\|^p_{W^m_p(\Omega_h)} := \sum_{E \in \Omega_h} \|v\|^p_{W^m_p(E)}\,,
\qquad 
&|v|^p_{W^m_p(\Omega_h)} := \sum_{E \in \Omega_h} |v|^p_{W^m_p(E)}\,,
\qquad & \text{if $1 \leq p < \infty$,}
\\
&\|v\|_{W^m_p(\Omega_h)} := \max_{E \in \Omega_h} \|v\|_{W^m_p(E)}\,,
\qquad 
&|v|_{W^m_p(\Omega_h)} := \max_{E \in \Omega_h} |v|_{W^m_p(E)}\,,
\qquad & \text{if $p = \infty$,}
\end{aligned}
\]
\end{itemize}
and the following polynomial projections:
\begin{itemize}
\item the $\boldsymbol{L^2}$\textbf{-projection} $\Pi_n^{0, E} \colon L^2(E) \to \Pk_n(E)$, given by
\begin{equation*}
\label{eq:P0_k^E}
\int_Eq_n (v - \, {\Pi}_{n}^{0, E}  v) \, {\rm d} E = 0 \qquad  \text{for all $v \in L^2(E)$  and $q_n \in \Pk_n(E)$,} 
\end{equation*} 
with obvious extension for vector functions $\boldsymbol{\Pi}^{0, E}_{n} \colon [L^2(E)]^2 \to [\Pk_n(E)]^2$;

\item the $\boldsymbol{H^1}$\textbf{-seminorm projection} ${\Pi}_{n}^{\nabla,E} \colon H^1(E) \to \Pk_n(E)$, defined by 
\begin{equation*}
\label{eq:Pn_k^E}
\left\{
\begin{aligned}
& \int_E \gr  \,q_n \cdot \gr ( v - \, {\Pi}_{n}^{\nabla,E}   v)\, {\rm d} E = 0 \quad  \text{for all $v \in H^1(E)$ and  $q_n \in \Pk_n(E)$,} \\
& \int_{\partial E}(v - \,  {\Pi}_{n}^{\nabla, E}  v) \, {\rm d}s= 0 \, ,
\end{aligned}
\right.
\end{equation*}
\end{itemize}
with global counterparts 
$\Pi_n^{0} \colon L^2(\Omega) \to \Pk_n(\Omega_h)$ and
${\Pi}_{n}^{\nabla} \colon H^1(\Omega_h) \to \Pk_n(\Omega_h)$
defined by
\begin{equation*}
\label{eq:proj-global}
(\Pi_n^{0} v)|_E = \Pi_n^{0,E} v \,,
\qquad
(\Pi_n^{\nabla} v)|_E = \Pi_n^{\nabla,E} v \,,
\qquad
\text{for all $E \in \Omega_h$.}
\end{equation*}

We finally mention one classical result for polynomials on star-shaped domains (see for instance \cite{brenner-scott:book}).

\begin{lemma}[Polynomial approximation]
\label{lm:bramble}
Under the assumption \textbf{(A1)}, for any $E \in \Omega_h$ and for any  smooth enough function $\phi$ defined on $E$, it holds 
\[
\begin{aligned}
&\|\phi - \Pi^{0,E}_n \phi\|_{W^m_p(E)} \lesssim h_E^{s-m} |\phi|_{W^s_p(E)} 
\qquad & \text{$s,m \in \N$, $m \leq s \leq n+1$, $p=1, \dots, \infty$,}
\\
&\|\phi - \Pi^{\nabla,E}_n \phi\|_{m,E} \lesssim h_E^{s-m} |\phi|_{s,E} 
\qquad & \text{$s,m \in \N$, $m \leq s \leq n+1$, $s \geq 1$,}
\\
&\|\nabla \phi - \boldsymbol{\Pi}^{0,E}_{n} \nabla \phi\|_{m,E} \lesssim h_E^{s-1-m} |\phi|_{s,E} 
\qquad & \text{$s,m \in \N$, $m+1 \leq s \leq n+1$, $s \geq 1$.}
\end{aligned}
\]
\end{lemma}


\subsection{Virtual Element spaces}
\label{sub:spaces}

Given a polygon $E$ and a positive integer $k$, we define the local ``enhanced'' virtual element space as
\begin{equation}
\label{eq:vem space}
\begin{aligned}
V_h(E) = 
\bigl\{
v_h \in H^1(E) \cap & C^0(\partial E) \quad \text{s.t.} \quad 
 v_h|_e  \in \Pk_k(e) \quad \text{for all $e \in \partial E$,} 
\bigr .
\\
\bigl .
& \Delta v_h \in \Pk_k(E) \,, \quad
(v - \PN v, \, \widehat{p}_k ) = 0 \quad 
\text{for all $\widehat{p}_k \in \Pk_k(E) / \Pk_{k-2}(E)$}
\bigr \} \,.
\end{aligned}
\end{equation}
For the finite dimensional space $V_h(E)$, one can check that the following linear operators are a set of DoFs:
\begin{itemize}
\item $\vek$: the pointwise values of $v_h$ at the vertexes of the polygon $E$,
\item $\eek$: the values of $v_h$ at $k-1$ internal points of a Gauss-Lobatto quadrature for every edge $e \in \partial E$,
\item $\pek$: the moments $\dfrac{1}{| E |} \int_E v_h \, m_{\alpha\beta} \, {\rm d} E\, $, $\forall m_{\alpha\beta} \in \mathcal{M}_{k-2}(E)$ where $\mathcal{M}_{k-2}(E)$ is the set of monomials defined as
\begin{equation}
\label{eq:internal-dofs}
\mathcal{M}_{k-2} \coloneqq \left\{
m_{\alpha\beta} \coloneqq \left( \dfrac{x - x_E}{h_E} \right)^\alpha
\left( \dfrac{y - y_E}{h_E} \right)^\beta
 \ \alpha,\beta\in\mathbb N\, , \alpha + \beta \leq k - 2
\right\}.
\end{equation}
\end{itemize}
Thanks to these DoFs, it is possible to compute the following projections:
\[
\PN \colon V_h(E) \to \Pk_k(E), \qquad
\P0 \colon V_h(E) \to \Pk_{k}(E), \qquad
\PP0P \colon \nabla V_h(E) \to [\Pk_{k}(E)]^2 \,.
\]
Gluing together the local spaces, we define the global virtual element space as
\begin{equation*}
\label{eq:vem global space}
V_h(\Omega_h) = \{v_h \in V \quad \text{s.t.} \quad v_h|_E \in V_h(E) \quad \text{for all $E \in \Omega_h$} \} \, ,
\end{equation*}
with the associated set of degrees of freedom:
\begin{itemize}
    \item $\vk$: the values of $v_h$ at the vertices;
    \item $\ek$: the values of $v_h$ at $k-1$ points on each edge $e$;
    \item $\pk$: the moments up to order $k-2$ for each element $E\in\Omega_h$.
\end{itemize}

We finally recall from \cite{cangiani:2017,brenner-sung:2018} the optimal approximation property for the space $V_h(\Omega_h)$.

\begin{lemma}[Approximation using virtual element functions]
\label{lm:interpolation}
Under the assumption \textbf{(A1)} for any $v \in V \cap H^{s+1}(\Omega_h)$ there exists $\vint \in V_h(\Omega_h)$ such that for all $E \in \Omega_h$ it holds
\[
\|v - \vint\|_{0,E} + h_E \|\nabla (v - \vint)\|_{0,E} \lesssim h_E^{s+1} |v|_{s+1,E} \, , 
\]
where $0 < s \le k$.
\end{lemma}

\subsection{Virtual Element Forms and the Discrete Problem}
\label{sub:forms}
We start observing that the bilinear forms 
$a(\cdot,\cdot)$ , 
$\bskew(\cdot,\cdot)$
and 
$c(\cdot,\cdot)$, see \eqref{eq:a-c}, \eqref{eq:b-c} and \eqref{eq:c-c}, 
can be obviously decomposed into local contributions
\begin{equation}
\label{eq:local-forms}
a(u,  v) \eqqcolon \sum_{E \in \Omega_h} a^E(u,  v) \, ,
\quad
\bskew(u,  v) \eqqcolon \sum_{E \in \Omega_h} \bskewE(u,  v) \, , 
\quad
c(u,  v) \eqqcolon \sum_{E \in \Omega_h} c^E(u,  v) \, .
\end{equation}
 
Using the DoFs introduced in Section \ref{sub:spaces}, we construct a computable counterpart of the above-mentioned forms, following the standard VEM procedure.

Hence, we define the bilinear form $a_h^E(\cdot,  \cdot) \colon V_h(E) \times V_h(E) \to \R$ as follows:
 \begin{equation*} \label{eq:aEh}
 \begin{split}
a_h^E(u_h,  v_h) &:=
\int_E \PZ0P \nabla u_h \cdot \PZ0P \nabla v_h \, {\rm d}E + 
\Stab\bigl((I - \PN) u_h, \, (I - \PN) v_h\bigr) \, .
\end{split}
\end{equation*}
%
%
Above, the stabilizing bilinear form $\Stab \colon V_h(E) \times V_h(E) \to \R$ is required to be computable and to satisfy
\begin{equation}
\label{eq:sEh}
\alpha_*|v_h|_{1,E}^2 \leq \Stab(v_h, v_h) \leq \alpha^* |v_h|_{1,E}^2 \, ,
\qquad \text{for all $v_h \in {\rm Ker}(\PN)$} \, ,
\end{equation}
for two positive uniform constants $\alpha_*$ and $\alpha^*$. 
In what follows, we choose the \texttt{dofi-dofi} stabilization (cf. \cite{volley, BLR:2017}, for instance), which is a common choice for VEM approach. 
Following \cite{BDLV:2021}, we replace the bilinear form $b^E(\cdot, \cdot) \colon V_h(E) \times V_h(E)  \to \R$ 
with $b_h^E(\cdot, \cdot)$, defined as
\begin{equation*}
\label{eq:bfb}
b_h^E(u_h, v_h) \coloneqq \int_E \bb \cdot \nabla \P0 u_h  \, \P0 v_h \, {\rm d}E + 
\int_{\partial E} (\bb \cdot \nn^E) (I - \P0) u_h  \, v_h \, {\rm d}s \, .
\end{equation*}
In the numerical scheme, we will employ the skew-symmetrized form (cf. \eqref{eq:b-c}):

\begin{equation*}
    \bskewEh(u_h , v_h)   = \frac{1}{2}\left( b_h^E(u_h, v_h) - b_h^E(v_h, u_h) \right) \, .
\end{equation*}
The reaction term is locally replaced by $c_h(\cdot , \cdot) \colon V_h(E) \times V_h(E)  \to \R$, defined as
\begin{equation*}
\label{eq:cE}
c_h^E(u_h, v_h) \coloneqq \int_E \P0 u_h  \, \P0 v_h \, {\rm d}E + 
|E| \, \Stab\bigl((I - \P0) u_h, \, (I - \P0) v_h\bigr) \, .
\end{equation*}
%
Following \cite{burman:2004, burman:2007}, we now introduce a VEM version of the local CIP-stabilization form,  
defined as
\begin{equation}
\label{eq:JhE}
J_h^E(u_h,v_h)
\coloneqq 
 \sum_{e \subset \partial E}  \dfrac{\gamma}{2} \int_e  \! \, h_e^2 \, [\nabla \Pg u_h] \cdot [\nabla \Pg v_h] \, {\rm d}s 
+
\gamma \, h_E \, \Stab\bigl((I - \PN) u_h, (I - \PN) v_h\bigr) \, ,
\end{equation}
where $[\nabla \cdot]$ denotes the gradient jump across $e$. If $e$ is a boundary edge we set $[\nabla \cdot ] = 0$. The parameter $\gamma$ is defined as
\[
\gamma (\partial E) \coloneqq \| \bb \cdot \nn^e \|_{L^\infty(\partial E)} \, ,
\]
where $\nn^e$ is one of the two outward normal vectors to $e$. Since we will work with the assumption $\| \bb \|_{[L^\infty(\Omega)]^2} = 1$, we will treat $\gamma$ as a constant. 

Moreover, we impose the Dirichlet boundary conditions by using a Nitsche-type technique. To this aim, we define the local forms:
\begin{equation*}\label{eq:NE}
\mathcal{N}_h^E(u_h,v_h) \coloneqq
- \epsilon  \langle \nabla \PN u_h \cdot \nn^E, v_h \rangle_{\Gamma_E} 
 + \dfrac{\epsilon}{\delta h_E} \left \langle u_h,v_h \right\rangle_{\Gamma_E}    
+ \dfrac{1}{2} \langle \vert \bb \cdot \nn \vert u_h, v_h \rangle_{\Gamma_E} \, ,
\end{equation*}
where $\Gamma_E = \partial E \cap \Gamma$, $\delta$ is a positive parameter to be chosen and $\langle \cdot , \cdot \rangle$ is the $L^2(\Gamma_E)$-scalar product.

\begin{remark}
The standard definition of Nitsche's method also considers a term
\[
- \epsilon \langle u_h, \nabla \PN v_h \cdot \nn^E\rangle_{\Gamma_E} \,  .
\]
Since we are not interested in achieving symmetry and in order to simplify the analysis of the method, we drop this term. Another difference with the standard formulation of Nitsche's method is the convective term. Usually, it is locally defined as
\[
- \langle (\bb \cdot \nn^E) u_h, v_h \rangle_{\Gamma_{\text{in}}\cap{\Gamma_E} } \, .
\]
By integration by parts, in the definition of 
$\bskew(\cdot,\cdot)$, we should consider also
\[
\dfrac{1}{2}\langle (\bb \cdot \nn^E) u_h,v_h \rangle_{\Gamma_E} \, .
\]
Summing the last two terms, we recover our definition of $\mathcal{N}_h^E(\cdot, \cdot)$.\qed
\end{remark}

Summing all of these contributions, we construct the discrete bilinear form $\Acip^E \colon V_h(E) \times V_h(E)  \to \R$ as
\begin{equation}
\label{eq:AcipE}
\Acip^E(u_h, v_h) = \epsilon a_h^E(u_h , v_h) + \bskewEh(u_h , v_h) + \sigma c_h^E(u_h , v_h) + \mathcal{N}_h^E(u_h,v_h) + J_h^E(u_h , v_h) \, ,
\end{equation}
and summing over all the polygons we obtain the global versions of the bilinear forms
\[
\begin{aligned}
&a_h(u_h, v_h) := \sum_{E \in \Omega_h} a_h^E(u_h, v_h)\,,
\qquad 
&\bskewh(u_h, v_h) := \sum_{E \in \Omega_h} \bskewEh(u_h, v_h)\,,
\\
&c_h(u_h, v_h) := \sum_{E \in \Omega_h} c_h^E(u_h, v_h)\,,
\qquad 
&J_h(u_h, v_h) := \sum_{E \in \Omega_h} J_h^E(u_h, v_h)\,,
\end{aligned}
\]
\[
\mathcal{N}_h(u_h, v_h) := \sum_{E \in \Omega_h} \mathcal{N}_h^E(u_h, v_h)\,,
\]
and
\begin{equation}
\label{eq:Acip}
\Acip (u_h, v_h) \coloneqq \sum_{E \in \Omega_h} \AcipE (u_h, v_h) \, .
\end{equation}
The discrete local and global load terms (here $\mathcal{F}_h^E \colon V_h(E) \to \R$) are 
\begin{equation}
\label{eq:FE}
\mathcal{F}_h^E(v_h) \coloneqq \int_E f \, \P0 v_h \, ,
\qquad 
\mathcal{F}_h(v_h) := \sum_{E \in \Omega_h} \mathcal{F}^E_h(v_h)  \,.
\end{equation}
Finally, the discrete problem reads as:
\begin{equation}
\label{eq:cip-vem}
\left \{
\begin{aligned}
& \text{find $u_h \in V_h(\Omega_h)$ s.t.} 
\\
& \Acip(u_h, \, v_h) = \mathcal{F}_h(v_h) \qquad \text{for all $v_h \in V_h(\Omega_h)$.}
\end{aligned}
\right.
\end{equation}

\subsection{Consistency of the method}
\label{ss:consistency}

Due to the polynomial projections entering in \eqref{eq:cip-vem}, it is easily seen that, as usual for the VEMs, the solution $u$ of the continuous problem \eqref{eq:problem-c} does not solve the discrete scheme \eqref{eq:cip-vem} (thus, strong consistency does not hold). However, if $u$ is more regular, say $u\in H^2(\Omega)\cap H^1_0(\Omega)$, then it holds:

\begin{equation}
\label{eq:consistency-1}
\Atcip(u , \, v_h) = \tilde{\mathcal{F}}(v_h) \qquad \text{for all $v_h \in V_h(\Omega_h)$} \, .
\end{equation}
where
\begin{equation*}
\label{eq:A-c}
\Atcip (u,v_h) 
\coloneqq
\sum_{E \in \Omega_h} \AtcipE (u, v_h) \, , \quad \tilde{\mathcal{F}}(v_h) \coloneqq\sum_{E \in \Omega_h} \tilde{\mathcal{F}}^E(v_h) \, ,
\end{equation*}
and the local forms are defined as follows.
\begin{flalign}
\label{eq:AE-c}
\bullet\qquad
\AtcipE(u, v_h) 
\coloneqq 
\epsilon \, a^E(u, v_h) 
+ \bskewE(u, v_h) 
+ \sigma \, c^E(u, v_h) 
+\tilde{\mathcal{N}}_h^E (u, v_h)
+ \tilde J_h^E(u, v_h) \, , &&
\end{flalign} 
with 
\begin{equation*}\label{eq:NE-c}
\tilde{\mathcal{N}}_h^E(u,v_h) \coloneqq
- \epsilon  \langle \nabla u \cdot \nn^E, v_h \rangle_{\Gamma_E}
+ \dfrac{\epsilon}{\delta h_E} \left \langle u,v_h \right\rangle_{\Gamma_E}    
+ \dfrac{1}{2} \langle \vert \bb \cdot \nn \vert u, v_h \rangle_{\Gamma_E} \, ,
\end{equation*}
where $\Gamma_E = \partial E \cap \Gamma$, and  
\[
\tilde J_h^E(u, v_h) 
\coloneqq 
\dfrac{1}{2} \sum_{e \subset \partial E} 
\gamma \int_e  \, h_e^2 \, [\nabla u] \cdot [\nabla v_h] \, {\rm d}s
= 
\dfrac{1}{2} \sum_{e \subset \partial E} 
\gamma \int_e \, h_e^2 \, [\nabla u \cdot \nn^e] [\nabla v_h \cdot \nn^e] \, {\rm d}s \, ;
\]

\begin{flalign}
\label{eq:FE-c}
\bullet\qquad
\tilde{\mathcal{F}}^E(v_h) \coloneqq \int_E f  \, v_h \, . &&
\end{flalign}

\section{Stability and convergence analysis}
\label{sec:theory}

\subsection{Preliminary results}
\label{sec:preliminary}

Before proving the stability of the discrete problem, we mention some preliminary results that are useful for our purposes. The first one is a standard inverse estimate for the virtual element functions.
\begin{lemma}[Inverse estimate]
\label{lm:inverse}
Under the assumption \textbf{(A1)}, for any $E \in \Omega_h$, there exists a uniform positive constant such that  
\[
| v_h |_{1,E} \lesssim h_E^{-1} \| v_h \|_{0,E} 
\quad \text{for all $v_h \in V_h(\Omega_h)$  .}
\]
\end{lemma}
We also recall, see \cite{BLR:2017,brenner-guan-sung:2017}, the following  inverse trace inequality.
\begin{lemma} [Inverse trace inequality]
\label{lm:invtrace}
Under the assumption \textbf{(A1)}, for any $E \in \Omega_h$ and for every $v_h \in V_h(E)$ such that 
$\Pi^{0,E}_{k-2} v_h \equiv 0$, 
it holds that
\[
\| v_h \|_{0,E} 
\lesssim 
h_E^{1/2} | v_h |_{0, \partial E} \, .
\]
\end{lemma}

We now construct a VEM version of the Oswald interpolation operator, see for instance \cite{burman:2004, burman:2007} for the FEM framework.
We consider a point $\nu$ associated to a DoF in $\ek$ or $\vk$ and 
we define $E_\nu \coloneqq \bigcup \{ E \in \Omega_h \quad \text{s.t} \quad \nu \in E\}$, i.e. the union of the set of all the elements that contain the point $\nu$. 
The quasi-interpolation operator $\pi$ for a sufficiently regular function $v$ is defined as
\begin{equation}
\label{eq:oswald-inter-1}
\pi v= \sum_{\nu \in \vk \cup \ek} \lambda_\nu (v) \phi_\nu + \sum_{E \in \Omega_h} \sum_{\alpha +\beta\le k-2} \mu_{\alpha\beta}^E(v) \varphi_{\alpha\beta}^E \, ,
\end{equation}
where $\{\varphi_\nu\}_{\nu \in \vk \cup \ek}$ are the canonical basis functions associated to the degree of freedom pointed at $\{\nu\}_{\nu \in \vk \cup \ek}$ and the coefficients $\{\lambda_\nu (v)\}$ are defined as
\begin{equation}
\label{eq:oswald-inter-2}
\lambda_\nu (v)\coloneqq  \dfrac{1}{| E_\nu|} \sum_{ E \subseteq E_\nu} v^E(\nu) \, |E| \, .
\end{equation}
Above, and from now on in this section, a superscript $E$ for a function denotes the restriction of that function to the element $E$. 
Similarly, above $\{ \varphi_{\alpha\beta}^E\}$ denote the basis functions associated to the degrees of freedom $\pek$, and $\{\mu_{\alpha\beta}^E(v)\}$ are the associated coefficients corresponding to $v$, defined as (cf. \eqref{eq:internal-dofs}):
\begin{equation}\label{eq:Oswald-interior}
\mu_{\alpha\beta}^E(v) = \dfrac{1}{| E |} \int_E v \, m_{\alpha\beta} \, {\rm d} E\, .
\end{equation}

We are ready to prove the following estimate concerning the interpolation error for piecewise polynomial functions. A FEM version of this result can be found in  \cite{burman:2004, burman:2007}.
\begin{proposition}
\label{prp:clmest}
Under assumption \textbf{(A1)}, for every $E \in \Omega_h$ it holds
\[
\| (I - \pi) p \|_{0,E} \lesssim h^{1/2} \sum_{e \in \mathcal{F}_E} \| [ p ] \|_{0,e} \qquad \text{for all $p \in \Pk_k(\Omega_h)$}\, ,
\]
where $\mathcal{F}_E \coloneqq \{ e \in \mathcal  E \quad \text{s.t} \quad e \cap \partial E \ne \emptyset \}$ is the set of the edges with at least one endpoint which is a vertex of $E$.
\end{proposition}

\begin{proof}
We introduce the difference
\[
\delta \coloneqq (I - \pi ) p \, .
\]
We restrict our attention to an element $E\in\Omega_h$, and consider $\delta^E$.
Since the DoFs in $\pek$ belong to one element only, we observe that for $\delta^E$ only the DoFs arising from $\vek$ and $\eek$ (i.e. the ones on the mesh skeleton), are involved. 
Hence, noting that $\delta^E \in V_h(E)$ and $\Pi^{0,E}_{k-2} \delta^E =0$ we can apply Lemma \ref{lm:invtrace}:
\begin{equation}
\label{eq:osw-1}
\| \delta^E \|_{0,E} \lesssim h^{1/2} \| \delta^E \|_{0,\partial E} \lesssim h \| \delta^E \|_{\infty,\partial E} \, .
\end{equation}
Since the basis function associated to $\vk$ and $\ek$ are scaled in a way that their $L^\infty-$norm is equal to 1, we have that
\begin{equation}
\label{eq:osw-2}
h \| \delta^E \|_{\infty,\partial E}
\lesssim h
\max_{\nu \in \eek \cup \vek}  | \delta^E (\nu) | \, .
\end{equation}
Exploiting the definition of the Oswald interpolant, we observe that if $\nu \in \eek$ is not on the boundary, we have that
\[
\begin{aligned}
\delta^E (\nu) =
p^E (\nu) - (\pi p)^E(\nu) &=
\dfrac{1}{| E \cup E' |}
\left(
| E \cup E' | \,p^E(\nu) - | E |\, p^E(\nu) - |E'|\, p^{E'}(\nu)
\right)\, \\
&=
c \, (p^E(\nu) - p^{E'}(\nu)) = c [p](\nu)\, , 
\end{aligned}
\]
where $E'$ is the second element that shares the node $\nu$. 
Thanks to the mesh assumptions {\bf (A1)}, all the values
\begin{equation*}
c 
=
\dfrac{| E \cup E' | -  | E |}{| E \cup E' |}  
=
\dfrac{| E' |}{| E \cup E' |} \approx \dfrac{1}{2} >0 \, .
\end{equation*}
are uniformly bounded from below and they do not depend on $h$; hence it holds
\begin{equation}
\label{eq:int-edge}
\max_{\nu \in \eek} | \delta^E (\nu) | \lesssim \max_{\nu \in \eek} | [p](\nu) | \, . 
\end{equation}
If $\nu \in \vek$, a similiar computation allows to bound $| \delta^E (\nu) |$ by means of the jumps of $p$ at the nodes on the edges containing $\nu$ (this set is denoted by $\mathcal{N}_\nu$ here below):
\begin{equation}
\label{eq:int-vertex}
 | \delta^E (\nu) | \lesssim \max_{\nu' \in \mathcal{N}_\nu} | [p](\nu') | \, . 
\end{equation}
Combining \eqref{eq:int-edge} and \eqref{eq:int-vertex}, we get
\begin{equation}
\label{eq:osw-3}
h\, \max_{\nu \in \eek \cup \vek}  | \delta^E(\nu) | \lesssim h\,  \max_{\nu\in e\, , e \in \mathcal{F}_E } | [p](\nu) | \lesssim h\, || [p] ||_{\infty,\mathcal{E}(E)}\, ,
\end{equation}
where $\mathcal{E}(E) \coloneqq \bigcup_{e\in \mathcal{F}_E }e$. Since an inverse estimate gives
\begin{equation}
\label{eq:osw-4}
h\, || [p] ||_{\infty,\mathcal{E}(E)} \lesssim h^{1/2} || [p] ||_{0,\mathcal{E}(E)} \lesssim 
h^{1/2} \sum_{e \in \mathcal{F}_E} \| [p] \|_{0,e}\, , 
\end{equation}
from \eqref{eq:osw-1}, \eqref{eq:osw-2}, \eqref{eq:osw-3} and \eqref{eq:osw-4} we obtain
\[
\| (I - \pi) p \|_{0,E}  = \| \delta^E \|_{0,E} \lesssim  h^{1/2} \sum_{e \in \mathcal{F}_E} \| [p] \|_{0,e}\, .
\]

\end{proof} 

\begin{lemma}
\label{lm:clmcont}
Under assumption \textbf{(A1)}, for every $E \in \Omega_h$ it holds
\[
\| \pi p \|_{0,E} \lesssim \| p \|_{0,\mathcal{D}(E)} \quad \text{for all $p \in \Pk_k(\Omega_h)$}\, ,
\]
where $\mathcal{D}(E) \coloneqq \bigcup \{ K \in \Omega_h \quad \text{s.t.}\quad \bar E \cap \bar K \ne \emptyset \}.$
\end{lemma}

\begin{proof}
Using triangular inequality, we obtain
\[
\| \pi p \|_{0,E} \leq \| p \|_{0,E} + \| (I-\pi) p \|_{0,E} \, .
\]
Thanks to Proposition \ref{prp:clmest}, we control the second term with the jumps
\[
\| (I-\pi) p \|_{0,E} 
\lesssim
h^{1/2} \sum_{e \in \mathcal{F}_E} \| [ p ] \|_{0,e}  
\, .
\]
Thanks to the polynomial trace inequality, we conclude
\[
\| (I-\pi) p \|_{0,E} \lesssim \| p \|_{0,\mathcal D (E)} \, ,
\]
hence
\[
\| \pi p \|_{0,E} \leq \| p \|_{0,E} + \| (I-\pi) p \|_{0,E} \lesssim \| p \|_{0,\mathcal D(E)} \, .
\]
\end{proof}


\subsection{Stability of the discrete problem}
\label{sec:infsup}

We start the theoretical analysis for the proposed method by introducing the local VEM-CIP norm 
\begin{equation} \label{eq:loc_nom_def}
\|v_h\|^2_{\cip  , E} 
:= 
\epsilon \, \| \nabla v_h \|^2_{0,E} + 
h \, \| \bb \cdot \nabla \P0 v_h \|^2_{0,E} + 
\sigma \, \| v_h \|^2_{0,E} + 
\Vert \xi (\epsilon, \bb) v_h \Vert^2_{0,\Gamma_E} +
J_h^E(v_h,v_h) \, ,
\end{equation}
where
\begin{equation}
\label{eq:xi-def}
\xi (\epsilon, \bb) 
\coloneqq
\left( 
\dfrac{\epsilon}{\delta h} + 
\dfrac{1}{2}\vert \bb \cdot \nn \vert
\right)^{1/2} \, ,
\end{equation}
with global counterpart
\begin{equation} \label{eq:glb_norm_def}
\|v_h\|^2_{\cip} \coloneqq \sum_{E \in \Omega_h} \|v_h\|^2_{\cip, E} \,.
\end{equation}


The following two lemmas will be useful to prove the stability of the method.

\begin{lemma}\label{lm:symm-part}
Under assumptions \textbf{(A1)}, given $v_h\in V_h(\Omega_h)$, it holds
\begin{equation}\label{eq:infsupsymmetry_gbl}
\Acip(v_h, v_h) \gtrsim \epsilon \| \nabla v_h \|_{0}^2 
+ 
J_h(v_h, v_h) 
+
\sigma \| v_h \|_{0}^2
+
\Vert \xi (\epsilon, \bb) v_h \Vert^2_{0,\Gamma}
 \, .    
\end{equation}
\end{lemma}

\begin{proof}
We proceed locally, on each $E\in \Omega_h$.
Thanks to the skew-symmetry property of 
$b_h^{skew,E}(\cdot,\cdot)$, testing the quadratic form 
$\AcipE(\cdot,\cdot)$ 
with $v_h$, we obtain 
\begin{equation} \label{eq:infsupsimmetry_first}
-\epsilon \langle \nabla \PN v_h \cdot \nn, v_h \rangle_{\Gamma_E}
+
\epsilon \| \nabla v_h \|_{0,E}^2 
+ 
J^E_h(v_h, v_h) 
+
\sigma \| v_h \|_{0,E}^2
+
\Vert \xi (\epsilon, \bb) v_h \Vert^2_{0,\Gamma_E}
\lesssim
\AcipE(v_h, v_h) \, .
\end{equation}
We now handle the non-symmetric first term in \eqref{eq:infsupsimmetry_first}. Thanks to the Cauchy-Schwarz inequality and the Young's inequality for a positive constant $\alpha$ to be chosen, we have that
\begin{multline*}
\epsilon \Vert \nabla v_h \Vert^2_{0,E}
-
\epsilon \langle \nabla \PN v_h \cdot \nn, v_h \rangle_{\Gamma_E}
+
\dfrac{\epsilon}{\delta h} \Vert v_h \Vert_{0,\Gamma_E}^2 \\
\geq
\epsilon \Vert \nabla v_h \Vert^2_{0,E}
-
 \dfrac{h \, \epsilon}{2 \alpha} \Vert \nabla \PN v_h \cdot \nn \Vert_{0,\Gamma_E}^2
+ 
\left( \dfrac{1}{\delta} - \dfrac{\alpha }{2} \right) \dfrac{\epsilon}{h} \Vert v_h \Vert^2_{0,\Gamma_E} \,. 
\end{multline*}
Using the polynomial trace inequality, under the assumptions \textbf{(A1)}, we have that
\[
h \Vert \nabla \PN v_h \cdot \nn \Vert_{0,\Gamma_E}^2 
\leq 
C_t \Vert \nabla \PN  v_h \Vert_{0,E}^2 
\leq 
C_t \Vert  \nabla v_h \Vert^2_{0,E} \, ,
\]
for a uniform positive constant $C_t$. Hence, if we set $\alpha = C_t$ and $0 < \delta < 2/C_t$, we obtain
\begin{equation*}
\epsilon \Vert \nabla v_h \Vert^2_{0,E}
-
\epsilon \langle \nabla \PN v_h \cdot \nn, v_h \rangle_{\Gamma_E}
+
\dfrac{\epsilon}{\delta h} \Vert v_h \Vert_{0,\Gamma_E}^2 
\gtrsim
\epsilon \Vert \nabla v_h \Vert^2_{0,E} +
 \dfrac{\epsilon}{\delta h} \Vert v_h \Vert^2_{0,\Gamma_E} \, . 
\end{equation*}
Inserting this in \eqref{eq:infsupsimmetry_first}, we obtain
\begin{equation*}\label{eq:infsupsimmetry_second}
\epsilon \| \nabla v_h \|_{0,E}^2 
+ 
J^E_h(v_h, v_h) 
+
\sigma \| v_h \|_{0,E}^2
+
\Vert \xi (\epsilon, \beta) v_h \Vert^2_{0,\Gamma_E}
\lesssim
\AcipE(v_h, v_h) \, .    
\end{equation*}
Summing over all to elements $E \in \Omega_h$, we get the control of the symmetric terms in 
$\| \cdot \|_\cip$:

\begin{equation}\label{eq:infsupsymmetry_gbl2}
\epsilon \| \nabla v_h \|_{0}^2 
+ 
J_h(v_h, v_h) 
+
\sigma \| v_h \|_{0}^2
+
\Vert \xi (\epsilon, \beta) v_h \Vert^2_{0,\Gamma}
\lesssim
\Acip(v_h, v_h) \, .    
\end{equation}

\end{proof}

\begin{lemma}\label{lm:adv-term}
Given $v_h\in V_h(\Omega_h)$, let us set
\begin{equation} \label{eq:wh}
w_h \coloneqq \pwh \, ,
\end{equation}
where $\bb_h$ is the $L^2$-projection of $\bb$ onto the space of piecewise linear functions $\Pk_1(\Omega_h)$. Then, under assumptions \textbf{(A1)}, if $\epsilon < h$ it holds
\begin{equation}\label{eq:adv-term-control}
\Acip(v_h, w_h) 
\geq
C_1\, h \| \bb \cdot \nabla \Pg v_h \|^2_{0,\Omega} 
-C_2\, \, \Acip(v_h, v_h) \, .
\end{equation}
\end{lemma}

\begin{proof}
Thanks to Lemma \ref{lm:clmcont}, we first notice that
\begin{equation}\label{eq:pi_estimate}
\| \pi (\bb_h \cdot \nabla \Pg v_h) \|_{0,E}
\lesssim
\| \bb_h \cdot \nabla \Pg v_h \|_{0,\mathcal D(E)} \, ,
\end{equation}
an estimate which will be frequently used in the sequel.

Recalling \eqref{eq:wh}, we locally have
\begin{equation}\label{eq:new_loc_A}
\begin{aligned}
\AcipE (v_h, w_h) 
& =
\epsilon \, a_h^E (v_h, \pwh) 
+ J_h^E (v_h, \pwh) 
 \\ 
& \quad +
\sigma \, c_h^E(v_h, \pwh)
+
\mathcal{N}^E_h(v_h, \pwh) \\
& \quad + 
\bskewEh (v_h, \pwh)\\
& = T_1 + T_2 + T_3 + T_4 + T_5\, .
\end{aligned}
\end{equation}
We consider each of the five terms in this equation.\\
{\it Estimate for} $\mathbf{(T_1).}$ 
Using Cauchy-Schwarz inequality, Lemma \ref{lm:inverse}, estimate \eqref{eq:pi_estimate} and recalling that $\epsilon < h$, we get
\begin{equation} \label{eq:ah_infsup}
\begin{aligned}
\epsilon \, a_h^E (v_h, \pwh) 
& \geq
- \epsilon \, a_h^E(v_h, v_h)^{1/2} \, a_h^E(\pwh, \pwh)^{1/2} \\
& \gtrsim 
-\epsilon^{1/2} \| \nabla v_h \|_{0,E} \, \epsilon^{1/2} \vert \pwh \vert_{1,E} \\
 & \gtrsim 
-\epsilon^{1/2} \| \nabla v_h \|_{0,E} \, \epsilon^{1/2} h^{-1} \| \pwh \|_{0,E} \\
%
& \gtrsim 
-\epsilon^{1/2} \| \nabla v_h \|_{0,E} \, h^{1/2} \| \bb_h \cdot \nabla \Pg v_h \|_{0,\mathcal{D}(E)}  \, .
\end{aligned}
\end{equation}
{\it Estimate for} $\mathbf{(T_2).}$
For the jump operator 
$J_h^E(\cdot, \cdot)$, 
we use again Cauchy-Schwarz inequality
\begin{equation*} \label{eq:J_partial}
J_h^E(v_h, \pwh) 
\geq
-J_h^E(v_h, v_h)^{1/2} \, J_h^E(\pwh, \pwh)^{1/2} \, .
\end{equation*}
Thanks to the trace inequality for polynomials, Lemma \ref{lm:inverse} and estimate \eqref{eq:pi_estimate}, we obtain ($w_h =\pwh$): 
\begin{equation} \label{eq:Jhinfsup}
\begin{aligned}
J_h^E(w_h, w_h)  
&=
\dfrac{\gamma}{2} \sum_{e \subset \partial E} \int_e \, h_e^2 \, [\nabla \Pg w_h]^2 \, {\rm d}s
+
\gamma \, h_E \, \Stab_J \bigl( ( I - \PN ) w_h, ( I - \PN ) w_h \bigr) \\
& \lesssim
 h \, \| \nabla \P0 w_h \|_{0,\mathcal{D}(E)}^2
+  h_E \, \vert \pwh \vert_{1,E}^2 \\
& \lesssim
  h^{-1} \, \|  \P0 w_h \|_{0,\mathcal{D}(E)}^2 +  h^{-1} \| \pwh \|_{0,E}^2 \\
& \lesssim
h^{-1} \| \pwh \|_{0,\mathcal{D}(E)}^2 \\
& \lesssim
  h \| \bb_h \cdot \nabla \Pg v_h \|_{0,\mathcal{D}(\mathcal{D}(E))}^2  \, ,
\end{aligned}
\end{equation}
where $\mathcal{D}(\mathcal{D}(E)) := \cup_{E'\subseteq \mathcal{D}(E)} \mathcal{D}(E')$.
Therefore, it holds 
\begin{equation} \label{eq:J_estimate}
J_h^E(v_h, \pwh) 
\gtrsim 
-J_h^E(v_h, v_h)^{1/2} \,   h^{1/2} \, \| \bb_h \cdot \nabla \Pg v_h \|_{0,\mathcal{D}(\mathcal{D}(E))}  \, .
\end{equation}
{\it Estimate for} $\mathbf{(T_3).}$
Using a similar procedure, we control the bilinear form
$c_h(\cdot,\cdot)$ 
in this way
\begin{equation} \label{eq:chinfsup}
\begin{aligned}
\sigma c_h^E(v_h, \pwh) 
&\gtrsim  
- \sigma \| v_h \|_{0,E} \, \| \pwh \|_{0,E} \\
&\gtrsim
- \| v_h \|_{0,E} \,  h^{1/2} \| \bb_h \cdot \nabla \Pg v_h \|_{0,\mathcal{D}(E)} \, . 
\end{aligned}
\end{equation}
where we used $h^{1/2} \lesssim 1$ to simplify later developments.

\noindent
{\it Estimate for} $\mathbf{(T_4).}$
For the Nitsche term, we have that
\[
\begin{split}
\mathcal{N}^E_h(v_h,w_h) & =
-  \epsilon \langle \nabla \PN v_h \cdot \nn^E, w_h \rangle_{\Gamma_E}  + \dfrac{\epsilon}{\delta h_E} \left \langle v_h, w_h \right\rangle_{\Gamma_E}    
+ \dfrac{1}{2} \langle \vert \bb \cdot \nn \vert v_h, w_h \rangle_{\Gamma_E} \, .
\end{split}
\]
We consider each of the three terms above.  Using Cauchy-Schwarz inequality, trace inequality, $\epsilon<h$ and inverse estimate, the first term is estimated by 
\begin{equation} \label{eq:etaN1}
\begin{split}
\epsilon \langle \nabla \PN v_h \cdot \nn^E, \pwh \rangle_{\Gamma_E}
&\gtrsim
- \epsilon  \, h^{-1/2} \Vert \nabla v_h \Vert_{0,E}  \, 
h^{-1/2} \Vert \pwh \Vert_{0,E} \\
&  \gtrsim 
-\epsilon^{1/2} \| \nabla v_h \|_{0,E} \, h^{1/2} \| \bb_h \cdot \nabla \Pg v_h \|_{0,\mathcal{D}(E)}  \, .
\end{split}
\end{equation}
For the second term we have 
\begin{equation}\label{eq:etaN2}
\begin{split}
\dfrac{\epsilon}{\delta h_E} \langle v_h, \pwh \rangle_{\Gamma_E}
&\gtrsim
-\dfrac{\epsilon}{\delta h} \Vert v_h \Vert_{0,\Gamma_E} \, h^{-1/2} \Vert \pwh \Vert_{0,E} \\
&\gtrsim
-\dfrac{\epsilon^{1/2}}{\delta h^{1/2}} \Vert v_h \Vert_{0,\Gamma_E} h^{1/2}\Vert  \pi ( \bb_h \cdot \nabla \Pg v_h) \Vert_{0,E} \\
&\gtrsim
 - \Vert \xi (\epsilon, \bb) v_h \Vert_{0,\Gamma_E}   h^{1/2} \| \bb_h \cdot \nabla \Pg v_h \|_{0,\mathcal{D}(E)}  \,  .
\end{split}
\end{equation}
For the last one, using the same estimates, we get 
\begin{equation}\label{eq:etaN3}
\begin{split}
\dfrac{1}{2} \langle \vert \bb \cdot \nn \vert v_h, \pwh \rangle_{\Gamma_E} 
& \gtrsim
- \Vert \xi (\epsilon, \bb) v_h \Vert_{0,\Gamma_E} \,   h^{1/2} \| \bb_h \cdot \nabla \Pg v_h \|_{0,\mathcal{D}(E)} \, .
\end{split}
\end{equation}
Hence it holds
\begin{equation}\label{eq:Nestimate}
    \mathcal{N}^E_h(v_h, w_h) \gtrsim - \left( \epsilon^{1/2} \Vert \nabla v_h \Vert_{0,E} + \Vert \xi (\epsilon, \bb) v_h \Vert_{0,\Gamma_E} \right)  h^{1/2} \Vert \bb_h \cdot \nabla \Pg v_h \Vert_{0,\mathcal{D}(E)} \, .  
\end{equation}
{\it Estimate for} $\mathbf{(T_5).}$ It is the most involved term.
The skew term $\bskewEh(v_h, w_h)$ is composed by two parts
\begin{equation} \label{eq:b+b}
\bskewEh(v_h, w_h) 
=
\dfrac{1}{2} ( b_h^E(v_h, w_h) - b_h^E(w_h, v_h)) \, ,
\end{equation}
and we consider each of these two terms separately. 
The first term is defined as
\begin{equation} \label{eq:bdef}
b_h^E(v_h, w_h) 
=
\bigl(\bb \cdot \nabla \P0 v_h, \P0 w_h \bigr)_{0,E} + 
\bigl((\bb \cdot \nn^E) ( I - \P0 ) v_h, \P0 w_h\bigr)_{0,\partial E} \, .    
\end{equation}
We split the first term of \eqref{eq:bdef} as
\begin{equation} \label{eq:b_split}
\begin{aligned}
\bigl( \bb \cdot \nabla \P0 v_h, \, \P0 w_h \bigr)_{0,E} 
& = 
\bigl( \bb \cdot \nabla \P0 v_h, \, w_h \bigr)_{0,E}  
+
\bigl( \bb \cdot \nabla \P0 v_h, \, ( \P0 - I ) w_h \bigr)_{0,E} \\
& =
\bigl( \bb \cdot \nabla \P0 v_h, \, h \bb_h \cdot \nabla \P0  v_h \bigr)_{0,E} \\
& \quad +        
\bigl( \bb \cdot \nabla \P0 v_h, \, w_h - h \bb_h \cdot \nabla \P0 v_h \bigr)_{0,E} \\
& \quad +        
\bigl(\bb\cdot \nabla\P0 v_h, \, (\P0 - I) w_h\bigr)_{0,E} \\
& \eqqcolon
\eta_{\bb_1} + \eta_{\bb_2} + \eta_{\bb_3} \, .
\end{aligned}
\end{equation}
We estimate each of these three quantities. 
For the first term we have
\begin{equation}\label{eq:bhinfsup1}
\begin{aligned}
\eta_{\bb_1} 
&=
(\bb \cdot \nabla \P0 v_h, \wh)_{0,E} \\
&=
h \, \| \bb \cdot \nabla \P0 v_h \|^2_{0,E}
+ 
(\bb \cdot \nabla \P0 v_h, h (\bb_h - \bb)\cdot \nabla \P0 v_h)_{0,E}\\
&\geq
h \, \| \bb \cdot \nabla \P0 v_h \|^2_{0,E}
- C\,h^{1/2} \| \bb \cdot \nabla \P0 v_h \|_{0,E} \, h^{1/2} | \bb |_{W^{1,\infty}(E)} h \| \nabla \P0 v_h \|_{0,E}\\
&\geq
h \, \| \bb \cdot \nabla \P0 v_h \|^2_{0,E}
- C\, h^{1/2} \| \bb \cdot \nabla \P0 v_h \|_{0,E} \, h^{1/2} | \bb |_{W^{1,\infty}(E)} \| v_h \|_{0,E}
\end{aligned}
\end{equation}
Recalling \eqref{eq:wh} and by Young's inequality we get:
\begin{equation}\label{eq:adv-eq-xi}
\begin{aligned}
\eta_{\bb_2} & = h \bigl ( \bb\cdot \nabla \P0 v_h, (\pi - I) (\bb_h \cdot \nabla \P0  v_h) \bigr)_{0,E} \\
& \geq
- \dfrac{h}{2} \| \bb\cdot \nabla \P0 v_h \|^2_{0,E} 
- \dfrac{h}{2} \| (\pi - I) (\bb_h \cdot \nabla \Pg v_h) \|^2_{0,E} \, .
\end{aligned}
\end{equation}
Since $\bb_h$ is piecewise linear, for the second term we can use Proposition \ref{prp:clmest} and obtain
\begin{equation}\label{eq:xi-eq}
\begin{aligned}
h \| (\pi - I) (\bb_h \cdot \nabla \Pg v_h) \|_{0,E}^2 
&\lesssim
h^2 \! \sum_{e \subset \mathcal{F}_E} \| [\bb_h \cdot \nabla \Pg v_h] \|_{0,e}^2 \\
& \lesssim 
h^2 \! \sum_{e \subset \mathcal{F}_E} \| [( \bb_h - \bb ) \cdot \nabla \Pg v_h] \|_{0,e}^2
+
h^2 \! \sum_{e \subset \mathcal{F}_E} \| [\bb \cdot \nabla \Pg v_h] \|_{0,e}^2 \\
& \lesssim 
h^2 \! \sum_{e \subset \mathcal{F}_E} \| [( \bb_h - \bb ) \cdot \nabla \Pg v_h] \|_{0,e}^2
+
J_h^{\mathcal{D}(E)} (v_h, v_h)
\end{aligned}
\end{equation}
On each $e$, we control the first term in the previous inequality as
\begin{equation}
\label{eq:xi-eq-2}
\begin{aligned}
h^2\| [(\bb_h -\bb)\cdot \nabla \Pg v_h] \|_{0,e}^2
& \lesssim 
h^4  | \bb |^2_{W^{1,\infty}(E\cup E')} h^{-1}\|  \nabla \Pg v_h  \|_{0,E\cup E'}^2  \\
& \lesssim
h  | \bb |^2_{W^{1,\infty}(E\cup E')} \|  \Pg v_h  \|_{0,E\cup E'}^2  \\
& \lesssim
h  | \bb |^2_{W^{1,\infty}(E\cup E')} \| v_h  \|_{0,E\cup E'}^2\, ,
\end{aligned}
\end{equation}
where $E$ and $E'$ are the two elements sharing the edge $e$.
Combining \eqref{eq:adv-eq-xi} with \eqref{eq:xi-eq} and \eqref{eq:xi-eq-2}, we obtain
\begin{equation} \label{eq:bhinfsup2}
\begin{aligned}
\eta_{\bb_2} 
& \geq - 
\dfrac{h}{2} \| \bb \cdot \nabla \P0 v_h \|^2_{0,E}
- C\Big(
h | \bb |^2_{W^{1,\infty}(\mathcal{D}(E))} \| v_h \|_{0,\mathcal{D}(E)}^2
+J_h^{\mathcal{D}(E)}(v_h, v_h) 
\Big) \, .
\end{aligned}
\end{equation}
It remains to control $\eta_{\bb_3}$. Since $\bb_h\in \Pk_1(E)$, it holds $\bigl( \bb_h \cdot \nabla \P0 v_h, (\P0 - I) w_h \bigr)_{0,E}= 0$. 

Hence we have 
\begin{equation} \label{eq:bhinfsup3}
\begin{aligned}
\eta_{\bb_3} & = \bigl( (\bb - \bb_h) \cdot \nabla \P0 v_h, (\P0 - I) w_h \bigr)_{0,E} \\ 
& \gtrsim 
- \| (\bb - \bb_h) \cdot \nabla \P0 v_h \|_{0,E} \, \| \pwh \|_{0,E} \\
& \gtrsim 
- | \bb |_{W^{1,\infty}(E)} h \| \nabla \P0 v_h \|_{0,E} \, h \| \bb_h \cdot \nabla \Pg v_h \|_{0,\mathcal{D}(E)}  \\
& \gtrsim
-    | \bb |_{W^{1,\infty}(E)} \| v_h \|_{0.\mathcal D(E)}^2
%
\end{aligned}
\end{equation}
Collecting \eqref{eq:bhinfsup1}, \eqref{eq:bhinfsup2} and \eqref{eq:bhinfsup3}, from \eqref{eq:b_split} we get 
\begin{equation}\label{eq:advt-term-est-E}
\begin{aligned}
\bigl( \bb \cdot \nabla \P0 v_h, \, \P0 w_h \bigr)_{0,E} \geq
\dfrac{h}{2} \, & \| \bb \cdot \nabla \P0 u_h \|^2_{0,E} 
- C\Big(
J_h^{\mathcal{D}(E)}(v_h, v_h) \\
&+ h^{1/2} \| \bb \cdot \nabla \P0 v_h \|_{0,E} \, h^{1/2} | \bb |_{W^{1,\infty}} \| v_h \|_{0,E}  \\
&    
+ h | \bb |^2_{W^{1,\infty}(\mathcal{D}(E))} \| v_h \|_{0,\mathcal{D}(E)}^2
+    | \bb |_{W^{1,\infty}(E)}  \| v_h \|_{0,\mathcal{D}(E)}^2
\Big) \, .
\end{aligned}
\end{equation}
Returning to \eqref{eq:bdef}, we have to control the boundary term. 
We first notice that, due to Agmon's inequality and Poincar\'e's inequality, it holds
\begin{equation*}
 \| ( I - \P0) v_h \|_{0,\partial E} \lesssim h^{1/2}| ( I - \P0) v_h |_{1,E}\, . 
\end{equation*}
Together with an inverse inequality for the polynomial $\P0 w_h$, the definition of $w_h$ (cf. \eqref{eq:wh}), and Lemma \ref{lm:clmcont}, we thus get: 
\begin{equation} \label{eq:bhinfsup4}
\begin{aligned}
\bigl( (\bb \cdot \nn^E) (I - \P0) v_h, \P0 w_h \bigr)_{0, \partial E}
& \gtrsim
-   \| ( I - \P0) v_h \|_{0,\partial E} \, \| \P0 w_h \|_{0,\partial E} \\
& \gtrsim
-   h^{1/2} | ( I - \P0) v_h |_{1,E} \, h^{-1/2}\| \P0 w_h \|_{0,E} \\
& \gtrsim
-   | ( I - \P0) v_h |_{1,E} \, \|  w_h \|_{0,E} \\
%
& \gtrsim 
-    \, J_h^E(v_h, v_h)^{1/2} \,  h^{1/2} \| \bb_h \cdot \nabla \Pg v_h \|_{0,\mathcal{D}(E)} \, .
\end{aligned}
\end{equation}
Above, we have also used the estimate, see \eqref{eq:JhE}: 
$$
h | ( I - \P0) v_h |_{1,E}^2 \lesssim J_h^E(v_h, v_h)\, .
$$
From \eqref{eq:bdef}, \eqref{eq:advt-term-est-E} and \eqref{eq:bhinfsup4} we get
\begin{equation}\label{eq:adv-est-E1}
\begin{aligned}
b_h^E(v_h, w_h)  \geq
\frac{h}{2} \, & \| \bb \cdot \nabla \P0 u_h \|^2_{0,E} 
- C\Big(
J_h^{\mathcal{D}(E)}(v_h, v_h) \\
&+ h^{1/2} \| \bb \cdot \nabla \P0 v_h \|_{0,E} \, h^{1/2} | \bb |_{W^{1,\infty}} \| v_h \|_{0,E}  \\
&    
+ h | \bb |^2_{W^{1,\infty}(\mathcal{D}(E))} \| v_h \|_{0,\mathcal{D}(E)}^2
+    | \bb |_{W^{1,\infty}(E)}  \| v_h \|_{0,\mathcal{D}(E)}^2 \\
& 
+    \, J_h^E(v_h, v_h)^{1/2} \,  h^{1/2} \| \bb_h \cdot \nabla \Pg v_h \|_{0,\mathcal{D}(E)}
\Big) \, .
\end{aligned}
\end{equation}
Finally, we need to control 
$- b_h^E(w_h, v_h)$, see \eqref{eq:b+b}. 
Integrating by parts, we obtain
\begin{equation} \label{eq:bwv}
\begin{aligned}
- b_h^E(w_h, v_h) 
& = - \bigl( \bb \cdot \nabla \P0 w_h, \P0 v_h \bigr)_{0,E}
- 
\bigl( (\bb \cdot \nn^E) (I - \P0) w_h, \P0 v_h \bigr)_{0,\partial E} \\
& =
\bigl( \bb \cdot \nabla \P0 v_h, \P0 w_h \bigr)_{0,E}
- 
\bigl( (\bb \cdot \nn^E)  w_h, \P0 v_h \bigr)_{0,\partial E} \\
& =
\bigl( \bb \cdot \nabla \P0 v_h, \P0 w_h \bigr)_{0,E}
- 
\bigl( (\bb \cdot \nn^E) w_h, 
(\P0 - I) v_h \bigr)_{0,\partial E} \\
& \quad - 
\bigl( (\bb \cdot \nn^E) w_h, 
v_h \bigr)_{0,\partial E} \, .
\end{aligned}
\end{equation}
The first two terms are similar to the case 
$b_h(v_h, w_h)$.
The last one vanishes on the interior edges when we sum over all $E\in\Omega_h$. Hence, we need to consider the elements $E$ sharing with $\partial\Omega$ at least an edge. Using Cauchy-Schwarz inequality, trace inequality, inverse estimates and the continuity of $\pi$, we obtain on these boundary edges 
\begin{equation}\label{eq:boundary-est}
\begin{aligned}
- \bigl( (\bb \cdot \nn^E) w_h, 
v_h \bigr)_{0,\partial E} 
&\geq 
- \Vert \xi (\epsilon, \bb) v_h \Vert_{0,\Gamma_E}
\Vert \xi (\epsilon, \bb) w_h \Vert_{0,\Gamma_E} \\
&\gtrsim
- 
\Vert \xi (\epsilon, \bb) v_h \Vert_{0,\Gamma_E}
h^{-1/2}\Vert w_h \Vert_{0,E} \\
&\gtrsim
- 
\Vert \xi (\epsilon, \bb) v_h \Vert_{0,\Gamma_E}
 h^{1/2} \| \bb_h \cdot \nabla \Pg v_h \|_{0,\mathcal{D}(E)}  \, .
\end{aligned}
\end{equation}
Therefore, from \eqref{eq:b+b}, \eqref{eq:adv-est-E1}, \eqref{eq:bwv} and \eqref{eq:boundary-est} we get
\begin{equation}\label{eq:bskewE}
\begin{aligned}
\bskewEh(v_h, w_h)
&\geq
\frac{h}{2} \,  \| \bb \cdot \nabla \P0 u_h \|^2_{0,E} - C\Big(
J_h^{\mathcal{D}(E)}(v_h, v_h) \\
&+ h^{1/2} \| \bb \cdot \nabla \P0 v_h \|_{0,E} \, h^{1/2} | \bb |_{W^{1,\infty}} \| v_h \|_{0,E}  \\
&    
+ h | \bb |^2_{W^{1,\infty}(\mathcal{D}(E))} \| v_h \|_{0,\mathcal{D}(E)}^2
+    | \bb |_{W^{1,\infty}(E)}  \| v_h \|_{0,\mathcal{D}(E)}^2 \\
& 
+ \big(    \, J_h^E(v_h, v_h)^{1/2} +  
\Vert \xi (\epsilon, \bb) v_h \Vert_{0,\Gamma_E} \big) \,  h^{1/2} \| \bb_h \cdot \nabla \Pg v_h \|_{0,\mathcal{D}(E)}
\Big) \, .
\end{aligned}
\end{equation}

We now consider the five local estimates \eqref{eq:ah_infsup}, \eqref{eq:J_estimate}, \eqref{eq:chinfsup}, \eqref{eq:Nestimate} and \eqref{eq:bskewE}. From \eqref{eq:new_loc_A}, summing over all the elements $E\in\Omega_h$, we obtain 
%
%
\begin{equation}\label{eq:almost-done}
\begin{aligned}
\Acip (v_h, w_h) 
& \geq
\frac{h}{2} \| \bb \cdot \nabla \Pg v_h \|_{0,\Omega}^2
- C \Big(
\sum_{E \in \Omega_h} \big( \epsilon^{1/2} \| \nabla v_h \|_{0,E}  \\
& \quad +
 J_h^E(v_h, v_h)^{1/2} +  \| v_h \|_{0,E} + \Vert \xi (\epsilon, \beta) v_h \Vert_{0,\Gamma_E} \big)    h^{1/2} \| \bb_h \cdot \nabla \P0 v_h \|_{0,E} \\
& \quad 
+  J_h(v_h, v_h)   +
\sum_{E \in \Omega_h} \left(h | \bb |^2_{W^{1,\infty}(E)}+    | \bb |_{W^{1,\infty}(E)}\right) \| v_h \|_{0,E}^2 \\
& \quad + 
\sum_{E \in \Omega_h}  h^{1/2} \| \bb_h \cdot \nabla \P0 v_h \|_{0,E} \, h^{1/2} | \bb |_{W^{1,\infty}(E)} \| v_h \|_{0,E}
\Big)\, .  
\end{aligned}
\end{equation}
Above, we have also used the property that, due to assumption \textbf{(A1)}, summing over the elements each polygon is counted only a uniformly bounded number of times, even when the terms involve norms on $\mathcal{D}(E)$ or $\mathcal{D}(\mathcal{D}(E))$.

We now notice that the triangular inequality, standard approximation results and an inverse estimate give
\begin{equation}\label{falcao}
h^{1/2} \| \bb_h \cdot \nabla \P0 v_h \|_{0,E}
\lesssim
h^{1/2}  \left( \| \bb \cdot \nabla \P0 v_h \|_{0,E}
+
| \bb |_{W^{1,\infty}(E)} \| v_h \|_{0,E} \right).
\end{equation}
Hence, from \eqref{eq:almost-done}, using also Young's inequality (with suitable constants) for the first and the last summations in the right-hand side, we get
\begin{equation*}
    \begin{aligned}
\Acip(v_h, w_h) 
\geq  C_1\, h \| \bb \cdot \nabla \Pg v_h \|_{0,\Omega}^2 - C_2\,\Big( \epsilon \| \nabla v_h \|_{0}^2 + J_h(v_h, v_h) +   \| v_h \|_{0}^2
+ \Vert \xi (\epsilon, \beta) v_h \Vert^2_{0,\Gamma}\Big) \, .  
    \end{aligned}
\end{equation*}
From Lemma \ref{lm:symm-part}, we now obtain
\[
\Acip(v_h, w_h) 
\geq
C_1\, h \| \bb \cdot \nabla \Pg v_h \|^2_{0,\Omega} 
-C_2\, \, \Acip(v_h, v_h) \, .
\]
   
\end{proof}

With Lemmas \ref{lm:symm-part} and \ref{lm:adv-term} at our disposal, the inf-sup condition easily follows.

\begin{proposition}\label{prp:inf-sup}
Under assumptions \textbf{(A1)}, it holds:
\begin{equation*} \label{eq:infsup}
\| v_h \|_{\cip} 
\lesssim 
\sup_{z_h \in V_h(\Omega_h)} \dfrac{\Acip (v_h, z_h)}{\| z_h \|_{\cip}}
\qquad \text{for all $v_h \in V_h(\Omega_h)$.}
\end{equation*}
\end{proposition}

\begin{proof}
We split the proof into two cases. \\
\emph{First case.} We first consider $\varepsilon < h$. Given $v_h \in V_h(\Omega_h)$, we take $z_h =w_h + \kappa v_h$, where $w_h$ is defined as in Lemma \ref{lm:adv-term}. From Lemmas \ref{lm:symm-part} and \ref{lm:adv-term}, for $\kappa$ sufficiently large we have  
\[
\Acip(v_h, z_h) =
\Acip(v_h, w_h + \kappa v_h) 
\gtrsim 
\| v_h \|_\cip^2 \, .
\]
In order to conclude the proof of the inf-sup condition, we have to prove the estimate
\[
\| w_h \|_\cip \lesssim \| v_h \|_\cip \, ,
\]
which obviously implies $\| z_h \|_\cip \lesssim \| v_h \|_\cip $. Recalling the norm definition \eqref{eq:loc_nom_def}-\eqref{eq:glb_norm_def} and that $w_h := \pwh$, the above continuity estimate follows from Lemma \ref{lm:clmcont}, estimate \eqref{eq:Jhinfsup} and observing that  

\begin{equation}\label{eq:cont-final}
h \| \bb \cdot \nabla \P0 w_h \|_{0,E}^2 
\lesssim 
  h^{-1} \| \P0 w_h \|_{0,E}^2
\lesssim
  h \| \bb_h \cdot \nabla \P0 v_h \|_{0,\mathcal{D}(E)}^2 \, ,   
\end{equation}
and
\begin{equation}\label{eq:cont-final2}
\begin{aligned}
\| w_h  \|_{0,\Gamma_E}^2  & = \| \pwh  \|_{0,\Gamma_E}^2 
 \lesssim  h\, \| \pi( \bb_h\cdot\nabla\Pi^{0}_k v_h ) \|_{0,E}^2 + 
 h^3\,| \pi(\bb_h\cdot\nabla\Pi^{0}_k v_h)  |_{1,E}^2\\
& \lesssim 
 h\,\| \pi( \bb_h\cdot\nabla\Pi^{0}_k v_h)  \|_{0,E}^2
 \lesssim h\, \| \bb_h\cdot\nabla\Pi^{0}_k v_h  \|_{0,\mathcal{D}(E)}^2
 \, .      
\end{aligned} 
\end{equation}

The above bounds \eqref{eq:cont-final} and \eqref{eq:cont-final2} are to be combined with \eqref{falcao}.

\noindent \emph{Second case.} 
We now consider the case $\varepsilon \ge h$.
In such case the proof simply follows from Lemma \ref{lm:symm-part} and the observation that
$$
h  \| \bb \cdot \nabla \P0 u_h \|^2_{0,E} 
\lesssim \varepsilon \| \nabla \P0 u_h \|^2_{0,E}
\lesssim \varepsilon \| \nabla u_h \|^2_{0,E} \, ,
$$
which allows to control also convection with $\Acip(v_h, v_h)$.

\end{proof}

\subsection{Error estimates}
\label{sub:error}

We begin our error analysis, which follows the steps of \cite{BDLV:2021}, with the following result. 

\begin{proposition}
\label{prp:abstract}
Let $u \in V$ and $u_h \in V_h(\Omega_h)$ be the solutions of problem \eqref{eq:problem-c} and problem \eqref{eq:cip-vem}, respectively. Furthermore, let us define 
\[
\eint \coloneqq u - \uint  ,
\]
where $\uint \in V_h(\Omega_h)$ 
is the interpolant function of $u$ defined in Lemma \ref{lm:interpolation}.
Then under assumption \textbf{(A1)}, it holds that
\begin{equation}
\label{eq:abstract}
\| u - u_h \|_{\cip} 
\lesssim 
\| \eint \|_{\cip}
+ 
\sum_{E \in \Omega_h} \bigl(
\errF + \erra + \errb + \errc + \errN + \errJ 
\bigr) \, ,
\end{equation}
where (cf. Section \ref{ss:consistency})
\[
\begin{aligned}
\errF &\coloneqq  \| \tilde{\mathcal{F}}^E - \mathcal{F}_h^E \|_\cipdual \, ,
\\
\erra &\coloneqq \epsilon \, \| a^E(u, \cdot) - a_h^E(\uint, \cdot) \|_\cipdual \,,
\\
\errb &\coloneqq \|\bskewE(u, \cdot) - \bskewEh(\uint, \cdot) \|_\cipdual \, ,
\\
\errc &\coloneqq   \| c^E(u, \cdot) -  c_h^E(\uint, \cdot) \|_\cipdual \,,
\\
\errN &\coloneqq \| \Tilde{\mathcal{N}}_h^E(u, \cdot) - \mathcal{N}_h^E(\uint, \cdot) \|_\cipdual \, ,
\\
\errJ &\coloneqq  \|\tilde{J}_h^E(u ,   \cdot) - J_h^E(\uint,   \cdot)\|_\cipdual  = \|J_h^E(\uint,   \cdot)\|_\cipdual\, , \\
\end{aligned}
\]
where $\|\cdot\|_\cipdual$ is the dual norm of $\|\cdot\|_\cip$.
\end{proposition}

\begin{proof}
We first introduce the following quantities
\[
\epi \coloneqq u - \Pi^{\nabla}_k u \, , \qquad
e_h \coloneqq u_h - \uint \, .
\]

Using triangular inequality, we have that
\[
\| u - u_h \|_{\cip} 
\leq 
\| u - u_{\mathcal{I}} \|_{\cip} 
+ \| u_{\mathcal{I}} - u_h \|_{\cip}
= 
\| e_\mathcal{I} \|_{\cip} 
+ \| e_h \|_{\cip} \, .
\]
Thanks to the inf-sup condition, and recalling that $u$ satisfies \eqref{eq:consistency-1}, we have that
\begin{equation*}
\begin{aligned}
\| e_h \|_{\cip} 
& \lesssim 
\sup_{v_h \in V_h(\Omega_h)} \dfrac{\Acip(e_h, v_h)} {\| v_h \|_{\cip}}
=
\sup_{v_h \in V_h(\Omega_h)} \dfrac{\Acip(u_h - u_\mathcal{I}, v_h)} {\| v_h \|_{\cip}} \\
& =
\sup_{v_h \in V_h(\Omega_h)} \dfrac{\mathcal{F}_h(v_h) - \Acip(u_\mathcal{I}, v_h)}{\| v_h \|_{\cip}} \\
& =
\sup_{v_h \in V_h(\Omega_h)} \dfrac{\mathcal{F}_h(v_h) - \tilde {\mathcal F}(v_h) + \Atcip (u, v_h) - \mathcal{A}_{\cip}(u_\mathcal{I}, v_h)}{\| v_h \|_{\cip}} \\
& =
\sup_{v_h \in V_h(\Omega_h)} \dfrac{\sum_{E \in \Omega_h}\bigl( \mathcal{F}_h^E(v_h) - \tilde{\mathcal F}^E(v_h) + \AtcipE (u, v_h) - \AcipE(u_\mathcal{I}, v_h) \bigr)} {\| v_h \|_{\cip}} \, .
\end{aligned}
\end{equation*}
Estimate \eqref{eq:abstract} easily follows by recalling the definitions of $\AtcipE$ and $\AcipE$, see \eqref{eq:AE-c} and \eqref{eq:AcipE}-\eqref{eq:Acip}.
\end{proof}

To properly bound all the terms in Proposition \ref{prp:abstract} we make the following assumptions:

\smallskip\noindent
\textbf{(A2) Data assumption.} The solution $u$, the advective field $\bb$ and the load $f$ in \eqref{eq:problem-c} satisfy:
\[
\begin{aligned}
&u \in H^{\reg+1}(\Omega_h)\,, &
&f \in H^{\reg+1}(\Omega_h)\,, &
&\bb \in [W^{\reg+1}_{\infty}(\Omega_h)]^2\,, &
\end{aligned}
\]
for some $0 < \reg \leq k$.

\begin{lemma}[Estimate of $\|\eint\|_{\cip}$]
\label{lm:epi}
Under assumptions \textbf{(A1)} and \textbf{(A2)},
the term $\|\eint\|^2_{\cip}$ can be bounded as follows (for $0 < s \le k$)
\begin{equation*}
\| \eint \|^2_{\cip,E}
\lesssim
\epsilon  \, h^{2\reg} \, \vert u \vert^2_{\reg+1,E} +  \, h^{2\reg+1} \, \vert u \vert^2_{\reg+1,E} \, .
\end{equation*}
\end{lemma}

\begin{proof}
By definition of $\| \cdot \|_\cip$, we have that
\begin{equation*}
\| \eint \|^2_{\cip,E} 
=
\epsilon \| \nabla \eint \|^2_{0,E}
+ h \| \bb \cdot \nabla \P0  \eint \|_{0,E}^2
+ \sigma \| \eint \|_{0,E}^2
+ \| \xi ( \epsilon, \bb) e_\mathcal{I} \|^2_{\Gamma_E}
+ J_h^E(e_\mathcal{I}, e_\mathcal{I}) \, .	
\end{equation*}
Using lemma \ref{lm:interpolation}, we have that
\begin{equation*}
\epsilon \| \nabla \eint \|^2_{0,E}
+ h \| \bb \cdot \nabla \P0  \eint  \|_{0,E}^2
\lesssim 
(\epsilon + h) \, \| \nabla \eint \|^2_{0,E} 
\lesssim
(\epsilon + h) \, h^{2\reg} \, \vert u \vert^2_{\reg+1,E} \, ,
\end{equation*}
and
\begin{equation*}
 \| \eint \|^2_{0,E} \lesssim h^{2\reg+2} \, | u |_{\reg+1,E}^2 \, .
\end{equation*}
For the Nitsche term we have that
\[
\| \xi ( \epsilon, \bb) e_\mathcal{I} \|^2_{\Gamma_E} = \dfrac{\epsilon}{\delta h} \langle  e_\mathcal{I},  e_\mathcal{I}\rangle_{\Gamma_E} + \langle | \bb \cdot \nn^E |  e_\mathcal{I},  e_\mathcal{I} \rangle_{\Gamma_E} \, .
\]
Using trace inequality and interpolation estimate, we obtain
\[
\dfrac{\epsilon}{\delta h} \langle  e_\mathcal{I},  e_\mathcal{I} \rangle_{\Gamma_E} 
\lesssim 
\dfrac{\epsilon}{\delta h^2} \| e_\mathcal{I} \|_{0,E}^2
+
\dfrac{\epsilon}{\delta} | e_\mathcal{I} |_{1,E}^2
\lesssim 
\epsilon \, h^{2s} \, |u|_{s+1,E}^2 \, ,
\]
and
\[
\langle | \bb \cdot \nn^E |  e_\mathcal{I},  e_\mathcal{I} \rangle_{\Gamma_E}
\lesssim
{\color{blue}   } \, h^{-1} \, \| e_\mathcal{I}\|_{0,E}^2
\lesssim
{\color{blue}   } \, h^{2s+1} \, |u|_{s+1,E}^2 \, .
\]
It remains to control the jump operator. We have 
\begin{equation*}
\begin{aligned}
J_h^E(e_\mathcal{I},e_\mathcal{I})
& =
\dfrac{\gamma}{2} \sum_{e \subset \partial E} \int_e \, h_e^2 \, [\nabla \Pg e_\mathcal{I}] \cdot [\nabla \Pg e_\mathcal{I}] \, {\rm d}s
+
\gamma \, h_E \, \Stab_J \, ( (I - \PN) e_\mathcal{I}, (I - \PN) e_\mathcal{I}) \\
& \lesssim
h^2 \, ( h^{-1} \| \nabla \Pg  e_\mathcal{I} \|_{0,\mathcal{D}(E)}^2 
+
h \, \vert \nabla \Pg  e_\mathcal{I} \vert_{1,\mathcal{D}(E)}^2) 
+
 h \, \vert (I - \PN) \eint \vert^2_{1,E} \, \\
& \lesssim
 h^2 \, (h^{-1} \| \nabla \Pg e_\mathcal{I} \|_{0,\mathcal{D}(E)}^2 )
+
 h \, \vert (I - \PN) \eint \vert^2_{1,E} \\
& \lesssim
 h \, \vert \eint \vert^2_{1,\mathcal{D}(E)} \lesssim  \, h^{2\reg+1} \, \vert u \vert_{\reg+1,\mathcal{D}(E)}^2 \, . 
\end{aligned}
\end{equation*}
\end{proof} 

\begin{lemma}[Estimate of $\errF$]
\label{lm:errF}
Under the assumptions \textbf{(A1)} and \textbf{(A2)},
the term $\errF$ can be bounded as follows (for $0 < s \le k$)
\begin{equation*}
\eta_\mathcal{F}^E 
\lesssim
h^{\reg+1} \, \vert f \vert_{\reg+1,E} \, .
\end{equation*}	
\end{lemma}

\begin{proof}
It is sufficient to follow the same procedure of \cite{BDLV:2021}. Using the orthogonality of $\P0$, Cauchy-Schwarz inequality, Poincaré inequality and  Lemma \ref{lm:bramble}, we obtain
\begin{equation*}
\begin{aligned}
\eta_{\mathcal{F}}^E 
& =
\tilde{\mathcal{F}}^E(v_h) - \mathcal{F}^E_h(v_h) \\
& =
\bigl(f, v_h - \P0 v_h \bigr)_{0, E}
=
\bigl( (I-\P0) f, (I - \P0) v_h \bigr)_{0, E} \\
& \leq
\| (I-\P0) f \|_{0,E} \, \| (I - \P0) v_h \|_{0, E} \\
& \lesssim
\| (I-\P0) f \|_{0,E} \, \| v_h \|_{0, E} \\
&\lesssim
h^{\reg+1} \, \vert f \vert_{\reg+1,E} \, \| v_h \|_{\cip,E} \, .  
\end{aligned}
\end{equation*}
\end{proof}

\begin{lemma}[Estimate of $\erra$]
\label{lm:erra}
Under the assumptions \textbf{(A1)} and \textbf{(A2)},
the term $\erra$ can be bounded as follows (for $0 < s \le k$)
\begin{equation*}
\erra 
\lesssim 
\epsilon^{1/2} \, h^{\reg} \, \vert u \vert_{\reg+1,\mathcal D(E)} \, .
\end{equation*}			
\end{lemma}
\begin{proof}
This result is proved following the line of Lemma 5.3 of \cite{BDLV:2021}. 
Adding and subtracting $\PN u$, using Cauchy-Schwarz inequality, we obtain
\begin{equation*}
\begin{aligned}
\erra
& =
\epsilon \, \tilde a^E_h(u, v_h) - \epsilon \, a_h^E(u_\mathcal{I}, v_h)
=
\epsilon \, \tilde a^E_h(u - \PN u,v_h) + \epsilon \, a_h^E(\PN u - u_\mathcal{I},v_h) \\
& \leq
\epsilon \, (\| \nabla e_\pi \|_{0,E} + (1 + \alpha^*) \| \nabla (\PN u - u_\mathcal{I}) \|_{0,E}) \, \| \nabla v_h \|_{0,E} \\
&\lesssim
\epsilon \, (\| \nabla e_\pi \|_{0,E} + \| \nabla e_\mathcal{I} \|_{0,E}) \, \| \nabla v_h \|_{0,E}
\lesssim
\epsilon^{1/2} \, h ^\reg \, \vert u \vert_{\reg+1,E} \, \| v_h \|_{\cip,E} \, .
\end{aligned}
\end{equation*}
\end{proof}

\begin{lemma}[Estimate of $\errb$]
\label{lm:errb}
Under the assumptions \textbf{(A1)} and \textbf{(A2)},
the term $\errb$ can be bounded as follows (for $0 < s \le k$)
\begin{equation*}
\errb 
\lesssim 
h^{\reg + 1/2} \| u \|_{\reg+1} 
+
\| \bb \|_{[W^{{\reg+1},\infty}]^2}  h^{\reg+1} \| u \|_{\reg+2,E}  
+ 
\int_{\partial E} (\bb \cdot \nn ^E) \, e_{\mathcal{I}} \, v_h \, {\rm d}s \, .
\end{equation*}
\end{lemma}

\begin{proof}
Recalling the definition, we need to estimate
\begin{equation*}
\begin{aligned}
\errbA
& \coloneqq 
\bigl( \bb \cdot \nabla u, v_h \bigr)_{0,E}
- 
\bigl( \bb \cdot \nabla \P0 u_\mathcal{I}, \P0 v_h \bigr)_{0,E} 
- 
\int_{\partial E} (\bb \cdot \nn ^E) \, (I - \P0) \, u_\mathcal{I} \, \P0 v_h \, {\rm d}s \, , \\
\errbB & \coloneqq 
\bigl( \P0 u_\mathcal{I}, \bb \cdot \nabla \P0 v_h \bigr)_{0, E}
- 
\bigl( u, \bb \cdot \nabla v_h \bigr)_{0,E} 
+ 
\int_{\partial E} (\bb \cdot \nn ^E) \, (I - \P0) \, v_h \, \P0 u_\mathcal{I} \, {\rm d}s \,  .
\end{aligned}
\end{equation*}
By integration by parts, we have
\begin{equation*}
\begin{aligned}
\eta_{b,A}^E 
& =
\bigl( \bb \cdot \nabla u, (I - \P0) \, v_h \bigr)_{0,E} 
+ 
\bigl( \bb \cdot \nabla (u - \P0 u_\mathcal{I}), \P0 v_h \bigr )_{0,E} \\
& \quad - 
\int_{\partial E} (\bb \cdot \nn^E) (I - \P0) \, u_\mathcal{I} \, \P0 v_h \, {\rm d}s \\
& =
\bigl( \bb \cdot \nabla u, (I - \P0) v_h \bigr)_{0,E} 
-
\bigl( u - \P0 u_\mathcal{I}, \bb \cdot \nabla \P0 v_h \bigr)_{0,E} \\
& \quad +
\int_{\partial E} (\bb \cdot \nn^E) (u - u_\mathcal{I}) \, \P0 v_h \,  {\rm d}s \\
& =
\bigl( (I - \P0) \bb \cdot \nabla u, (I - \P0) v_h \bigr)_{0,E}  
+ 
\bigl( \P0 u_\mathcal{I} - u, \bb \cdot \nabla\P0 v_h \bigr)_{0,E} \\
& \quad +
\int_{\partial E} (\bb \cdot \nn^E) \, e_\mathcal{I} \, \P0 v_h \, {\rm d}s \\
& \eqqcolon  
\eta_{b,1}^E + \eta_{b,2}^E + \eta_{b,3}^E \, ,
\end{aligned}
\end{equation*}
and 
\begin{equation*}
\begin{aligned}
\eta_{b,B}^E 
& =
\bigl( \P0 u_\mathcal{I} - u, \bb \cdot \nabla \P0 v_h \bigr)_{0,E} 
-
\bigl( u, \bb \cdot \nabla (I - \P0) v_h \bigr)_{0,E} \\
& \quad +
\int_{\partial E} (\bb \cdot \nn^E) \, (I - \P0) \, v_h \, \P0 u_\mathcal{I} \, {\rm d}s \\
& = 
\bigl( \P0 u_\mathcal{I} - u, \bb \cdot \nabla \P0 v_h \bigr)_{0,E} 
+
\bigl( \bb \cdot \nabla u, (I - \P0) v_h \bigr)_{0,E} \\
& \quad +
\int_{\partial E} (\bb \cdot \nn^E) \, (I - \P0) \, v_h \, (\P0 u_\mathcal{I} - u) \, {\rm d} s \\
& =
\bigl(\P0 u_\mathcal{I} - u, \bb \cdot \nabla \P0 v_h\bigr)_{0,E} 
+
\bigl((I - \P0)\bb \cdot \nabla u,(I - \P0)v_h\bigr)_{0,E}  \\
& \quad +
\int_{\partial E} (\bb \cdot \nn^E) \, (I - \P0) \, v_h \, (\P0 u_\mathcal{I} - u) \, {\rm d} s \\
& \eqqcolon 
\eta_{b,2}^E + \eta_{b,1}^E + \eta_{b,4}^E \, .
\end{aligned}
\end{equation*}
yielding the following expression for 
$\eta_{b}^E$
\begin{equation} \label{E: sumb}
2 \eta_{b}^E = 2\eta_{b,1}^E + 2\eta_{b,2}^E + \eta_{b,3}^E + \eta_{b,4}^E \, .
\end{equation}
We now analyze each term in the sum above.

\noindent
$\bullet \,\, \eta_{b,1}^E$: using Cauchy-Schwarz, the continuity in $\P0$ in $L^2$ and standard estimates, we obtain 
\begin{equation*}
\begin{aligned}
\eta_{b,1}^E 
& =
\bigl( (I - \P0) \bb \cdot \nabla u, (I - \P0) v_h \bigr)_{0,E} \\
& \leq
\| (I-\P0) \bb \cdot \nabla u\|_{0,E} \,
\| v_h \|_{0,E} \\
&\leq
\| (I - \P0) \bb \cdot \nabla u \|_{0,E} \, \| v_h \|_{\cip,E} \\
&\lesssim
h^{\reg+1} \vert \bb \cdot \nabla u \vert_{\reg+1,E}\| v_h \|_{\cip,E}  \\
&\lesssim
h^{\reg+1} \, \| u \|_{\reg+1,E} \,
\| \beta \|_{[W^{\reg+1}_\infty(E)]^2} \, \| v_h \|_{\cip,E} .
\end{aligned}
\end{equation*}
\noindent
$\bullet \,\, \eta_{b,2}^E$: we have that
\begin{equation*}
%
%
%
%
%
\begin{aligned}
\eta_{b,2}^E 
& =
\bigl( \P0 u_\mathcal{I} - u, \bb \cdot \nabla \P0 v_h \bigr)_{0,E} \\
& \leq
\| \P0 u_\mathcal{I} - u \|_{0,E} \,  \| \bb \cdot \nabla \P0 v_h \|_{0,E} \\
& \leq 
\bigl( \| (I - \P0)u \|_{0,E} + \| e_{\mathcal{I}} \|_{0,E} \bigr) 
\| \bb \cdot \nabla \P0 v_h \|_{0,E} \\
& \lesssim h^{\reg + 1/2} \| u \|_{\reg+1} \| v_h \|_{\cip,E} \, .  
\end{aligned}
\end{equation*}


\noindent
$\bullet \,\, {\eta^E_{b,3} + \eta^E_{b,4}}$: we use a scaled trace inequality making use of the scaled norm
\begin{equation*}
\forall v \in H^1(E), \quad \tri v \tri_{1,E}^2 \coloneqq \| v \|_{L^2(E)}^2 + h^2_E \, \vert v \vert^2_{H^1(E)} \, .
\end{equation*}
We obtain
\begin{equation*}
\begin{aligned}
\eta^E_{b,3} + \eta^E_{b,4}
& =
\int_{\partial E} (\bb \cdot \nn^E) \, e_\mathcal{I} \, \P0 v_h \, {\rm d} s
+
\int_{\partial E} (\bb \cdot \nn^E) \, (I - \P0) \, v_h \, (\P0 u_\mathcal{I} - u) \, {\rm d} s \\
& = 
\int_{\partial E} (\bb \cdot \nn^E) \, (\P0 - I) \, v_h \, (e_\mathcal{I} + u - \P0 u_\mathcal{I}) \, {\rm d} s
+
\int_{\partial E} (\bb \cdot \nn^E) \, e_\mathcal{I} \, v_h \, {\rm d} s\\
& \lesssim
   \, (
\| e_\mathcal{I} \|_{L^2(\partial E)}
+
\| u - \P0 u_\mathcal{I} \|_{L^2(\partial E)}
)  \, 
\| (I - \P0)v_h\|_{L^2(\partial E)} \\
&  \quad +
\int_{\partial E} (\bb \cdot \nn^E) \, e_\mathcal{I} \, v_h \, {\rm d} s\\
& \lesssim
   \, h_E^{-1} \, (
\tri e_\mathcal{I} \tri_{1,E}
+
\tri u - \P0 u_\mathcal{I} \tri_{1,E}
)  \, 
\| (I - \P0)v_h\|_{0,E} \\
& \quad +
\int_{\partial E} (\bb \cdot \nn^E) \, e_\mathcal{I} \, v_h {\rm d} s\\
& 
\lesssim
   h^{-1/2} (
\tri e_\mathcal{I} \tri_{1,E}
+
\tri u - \P0 u_\mathcal{I} \tri_{1,E}
)  \, 
h^{1/2} \| \nabla(I - \PN)v_h\|_{0,E} 
+
\int_{\partial E} (\bb \cdot \nn^E) \, e_\mathcal{I} \, v_h {\rm d} s \\
& 
\lesssim
{  } h^{\reg+1/2} \vert u \vert_{\reg+1,E} \| v_h \|_{\cip,E} + \int_{\partial E} (\bb \cdot \nn^E) \, e_\mathcal{I} \, v_h \, {\rm d} s \, ,
\end{aligned}
\end{equation*}
where in the last step we used the $J_h(v_h,v_h)$ term in the definition of $\| v_h \|_{\cip,E}$.

The thesis now follows gathering the last three inequalities in \eqref{E: sumb}.
\end{proof}

\begin{lemma}[Estimate of $\errc$]
\label{lm:errc}
Under the assumptions \textbf{(A1)} and \textbf{(A2)},
the term $\errc$ can be bounded as follows (for $0 < s \le k$)
\begin{equation*}
\errc 
\lesssim 
 h^{\reg + 1} \, \vert u \vert_{\reg+1,E}  \,
\end{equation*}
\end{lemma}

\begin{proof}
Similarly to Lemma \ref{lm:erra}, we have that
\begin{equation*}
\begin{aligned}
\errc
&=
\tilde c^E_h(u,v_h) -  c_h^E(u_\mathcal{I},v_h)
=
 \tilde c^E_h(u - \P0 u,v_h) + c_h^E(\P0 u - u_\mathcal{I},v_h) \\
&\leq
 (\| e_\pi \|_{0,E} + (1 + \alpha^*) \| \P0 u - u_\mathcal{I}\|_{0,E}) \, \| v_h \|_{0,E}\\
&\lesssim
 (\|  e_\pi \|_{0,E} + \|  e_\mathcal{I} \|_{0,E}) \, \| v_h \|_{0,E}
\lesssim
 h^{\reg + 1} \, \vert u \vert_{\reg+1,E} \, \| v_h \|_{\cip,E} \, .
\end{aligned}
\end{equation*}

\end{proof}

\begin{lemma}[Estimate of $\errN$]
\label{lm:errN}
Under the assumptions \textbf{(A1)} and \textbf{(A2)},
the term $\errN$ can be bounded as follows (for $0 < s \le k$)
\begin{equation*}
\errN
\lesssim 
(\epsilon^{1/2} h^s  +  h^{s+1/2} )\vert u \vert_{s+1,E} \, .
\end{equation*}

\begin{proof}
By definition of the two bilinear forms, we have that
\[
\begin{split}
\errN 
&=
- \epsilon \langle \nabla u \cdot \nn ,  v_h \rangle_{\Gamma_E} + \epsilon  \langle \nabla \PN u_{\mathcal{I}} \cdot \nn , v_h  \rangle_{\Gamma_E} \\
& + 
\dfrac{\epsilon}{\delta h_E} \langle u, v_h \rangle_{\Gamma_E} 
- \dfrac{\epsilon}{\delta h_E} \langle u_{\mathcal{I}}, v_h \rangle_{\Gamma_E} 
- \dfrac{1}{2} \langle |\bb \cdot \nn| u, v_h \rangle_{\Gamma_E} 
+ \dfrac{1}{2} \langle |\bb \cdot \nn| u_{\mathcal{I}}, v_h \rangle_{\Gamma_E} \\
& \eqqcolon
\eNa + \eNb + \eNc  \, .
\end{split}
\]
Now, we estimate each of the three terms. Using trace inequality, the first returns 
\[
\begin{split}
\eNa 
&=
- \epsilon \langle \nabla u \cdot \nn ,  v_h \rangle_{\Gamma_E} + \epsilon  \langle \nabla \PN u_{\mathcal{I}} \cdot \nn , v_h  \rangle_{\Gamma_E}  \\
&=
- \epsilon \langle \nabla (u - \nabla \PN u_{\mathcal{I}}) \cdot \nn , v_h \rangle_{\Gamma_E} \\
&\lesssim 
\epsilon ( h^{-1/2} \Vert \nabla u - \nabla \PN u_{\mathcal{I}} \Vert_{0,E} + h^{1/2} | \nabla u - \nabla \PN u_{\mathcal{I}} |_{1,E} ) \Vert v_h \Vert_{\Gamma_E} \\
&\lesssim
\epsilon^{1/2} (\Vert \nabla u - \nabla \PN u_{\mathcal{I}} \Vert_{0,E} +  h | \nabla u - \nabla \PN u_{\mathcal{I}} |_{1,E} ) \Vert v_h \Vert_{\text{cip},E} \, .
\end{split}
\]
Adding and subtracting $\PN u$, using triangular inequality and Lemma \ref{lm:bramble}, we obtain
\[
\eNa \lesssim \epsilon^{1/2} h^s \vert u \vert_{s+1,E} \Vert v_h \Vert_{\text{cip},E} \, .
\]
For the second term, using trace inequality and interpolation estimate, we have that
\[
\begin{split}
\eNb 
&=
\dfrac{\epsilon}{\delta h} \langle u, v_h \rangle_{\Gamma_E} - \dfrac{\epsilon}{\delta h_E} \langle u_{\mathcal{I}}, v_h \rangle_{\Gamma_E} \\
& \lesssim
\dfrac{\epsilon}{\delta h} \Vert u - u_{\mathcal{I}}\Vert_{\Gamma_E} \Vert \Vert v_h \Vert_{\Gamma_E} \lesssim
\left( \dfrac{\epsilon}{\delta h} \right)^{1/2} \Vert u - u_{\mathcal{I}}\Vert_{\Gamma_E} \Vert \Vert v_h \Vert_{\cip,E} \\
&\lesssim
\epsilon^{1/2} h^s \vert u \vert_{s+1,E} \Vert v_h \Vert_{\text{cip},E} \, .
\end{split}
\]
Finally, the last one is treated in a very similar way with respect to the previous one, it gives
\begin{equation}\label{eq:boundarycontr}
\begin{split}
\eNc 
&=
-\dfrac{1}{2} \langle |\bb \cdot \nn| u, v_h \rangle_{\Gamma_E} - \dfrac{1}{2} \langle |\bb \cdot \nn| u_{\mathcal{I}}, v_h \rangle_{\Gamma_E} \\
&\lesssim
 h^{-1/2} \Vert u - u_{\mathcal{I}} \Vert_{0,E} \Vert v_h \Vert_{\cip,E} \\
& \lesssim  h^{s+1/2} | u |_{s+1,E} \Vert v_h \Vert_{\cip,E} \, .
\end{split}
\end{equation}
\end{proof}
\end{lemma}

\begin{lemma}[Estimate of $\errJ$]
\label{lm:errJ}
Under the assumptions \textbf{(A1)} and \textbf{(A2)},
the term $\errJ$ can be bounded as follows (for $0 < s \le k$)
\begin{equation*}
\errJ 
\lesssim 
h^{s + 1/2} \, | u |_{s+1,E} \, .
\end{equation*}
\end{lemma}
\begin{proof} Using Cauchy-Schwarz inequality, we have that
\begin{equation}
\begin{aligned}
J^E_h(u_{\mathcal{I}}, v_h) 
& \leq 
J^E_h(u_{\mathcal{I}}, u_{\mathcal{I}})^{1/2} J^E_h(v_h, v_h)^{1/2} \\
& \leq
J^E_h(u_{\mathcal{I}}, u_{\mathcal{I}})^{1/2} \| v_h \|_{\cip,E} \, .
\end{aligned}
\end{equation}
Since the solution $u$ is sufficiently smooth, we have that
\begin{equation*}
\begin{aligned}
J^E_h(u_{\mathcal{I}}, v_h)
& = 
\sum_{e \subset \partial E}\int_{e} \dfrac{\gamma}{2} \, h^2_e \, [\nabla \Pg  u_\mathcal{I}]\cdot[\nabla \Pg  u_\mathcal{I}]
+
h_E \, \gamma \, \mathcal{S}_j^E( ( I - \PN) u_\mathcal{I},( I - \PN) u_\mathcal{I})\\
& = 
\sum_{e \subset \partial E} \int_{e} \dfrac{\gamma}{2} \, h^2_e \, [\nabla (\Pg  u_\mathcal{I} -  u)]^2 
+
h_E \, \gamma \, \mathcal{S}_j^E( ( I - \PN) u_\mathcal{I},( I - \PN) v_h)\\
& \lesssim
 \sum_{E' \in \mathcal D(E)}  h^2 \, \| \nabla (\Pg  u_\mathcal{I} -  u) \|_{0,\partial E'}^2 
+   h \, | ( I - \PN) u_\mathcal{I} |_{1,E}^2 \, .
\end{aligned}
\end{equation*}
Using trace inequality, we obtain for the first term
\begin{equation*}
\begin{aligned}
\| \nabla \Pg  u_\mathcal{I} - \nabla u \|_{0,\partial E'}
&\lesssim
\bigl( h^{-1} \, \| \nabla \Pg  u_\mathcal{I} - \nabla u \|^2_{0,E'} + h \, \vert \nabla \Pg  u_\mathcal{I} - \nabla u \vert^2_{1,E'} \bigr)^{1/2} \, .
\end{aligned}
\end{equation*}
Adding and subtracting $\nabla \P0  u$, using Lemma \ref{lm:bramble} and interpolation estimate, we obtain
\begin{equation*}
\begin{aligned}
h^{-1/2} \| \nabla \P0  u_\mathcal{I} - \nabla u\|_{0,E'}
&\lesssim
h^{-1/2} \| \nabla \P0  u - \nabla u \|_{0,E'} + h^{-1/2}\| \nabla \P0  u_\mathcal{I} - \nabla \P0  u \|_{0,E'} \\
&\lesssim
h^{\reg - 1/2} \vert u \vert_{\reg + 1 ,E'} \, , 
\end{aligned}
\end{equation*}
and similarly, we have that
\begin{equation*}
\begin{aligned}
h^{1/2} \vert \nabla \P0  u_\mathcal{I} - \nabla u\vert_{1,E'}
&\lesssim
h^{\reg - 1/2} \vert u \vert_{\reg + 1 ,E'} \, .
\end{aligned}
\end{equation*}
Using Lemma \ref{lm:bramble}, we have that
\begin{equation*}
\begin{aligned}
 h^{1/2} \, | ( I - \PN) u_\mathcal{I} |_{1,E}
& \lesssim
 h^{1/2} \, ( \| \nabla \eint \|_{0,E} + \| \nabla \epi \|_{0,E} ) 
\lesssim
h^{\reg + 1/2} | u |_{\reg + 1, E} \, .
\end{aligned}
\end{equation*}
We conclude
\begin{equation*}
J^E_h(u_{\mathcal{I}}, v_h) 
\lesssim 
 h^{s + 1/2} \, | u |_{s+1,\mathcal D(E)} \, \| v_h \|_{\cip, \mathcal D(E)} \, .
\end{equation*}
\end{proof}

We thus have the following proposition.

\begin{proposition}
\label{prp:bfb}
Under the assumptions \textbf{(A1)} and \textbf{(A2)}, let $u \in V$ be the
solution of equation \eqref{eq:problem-c} and $u_h \in V_h(\Omega_h)$ be the solution of equation \eqref{eq:cip-vem}. 
Then it holds that
\[
\begin{aligned}
\|u - u_h\|^2_{\cip} \lesssim
\sum_{E \in \Omega_h} \Theta^E
\left(
\epsilon \, h^{2\reg} 
+   
h^{2\reg+1} 
\right) \, ,
\end{aligned}
\]
where the constant $\Theta^E$ depends on
$\|u\|_{s+2,E}$, $\|f\|_{s+1,E}$, $\|\beta\|_{[W^{s+1}_{\infty}(E)]^2}$.
\end{proposition}

\begin{proof}
It it sufficient to use Proposition \ref{prp:abstract} combined with Lemmas \ref{lm:epi}, \ref{lm:errF}, \ref{lm:erra}, \ref{lm:errb}, \ref{lm:errc}, \ref{lm:errN} \ref{lm:errJ}, noting that 
\[
\sum_{E \in \Omega_h} \int_{\partial E\setminus\partial\Omega} (\bb \cdot \nn^E) \, e_\mathcal{I} \, v_h {\rm d} s = 0 \, ,
\]
and the contributions stemming from $\partial\Omega$ are controlled as in \eqref{eq:boundarycontr}.
\end{proof}

\subsection{A special case: advection-diffusion problem with $\bb \in \Pk_1(\Omega)$}\label{ss:p1beta}

We consider problem \eqref{eq:problem-c} in a particular situation: we assume an advection term $\bb \in \Pk_1(\Omega)$, i.e. globally linear, and we allow the reaction coefficient $\sigma = 0$. We do not make further assumptions on the diffusion coefficient $\epsilon$ and on the load term $f$. Thus, the advection-diffusion problem reads as (cf. \eqref{eq:problem-c})
\begin{equation}\label{eq:problem-c_special}
\left\{
\begin{aligned}
& \text{find $u \in V$ such that:} 
\\
& \epsilon a(u,  v) + \bskew (u, v)  = \tilde{ \mathcal F }(v) \qquad \text{for all $v \in V$.}
\end{aligned}
\right.
\end{equation}
In this case, even without the reaction term, we are able to prove robust estimates for the approximation of problem \eqref{eq:problem-c_special}. 
Using the same approach as before, the discrete version of problem \eqref{eq:problem-c_special} reads as
\begin{equation}
\label{eq:cip-vem_special}
\left \{
\begin{aligned}
& \text{find $u_h \in V_h(\Omega_h)$ such that:} 
\\
& \Acip^{\rm ad}(u_h, \, v_h) = \mathcal{F}_h(v_h) \qquad \text{for all $v_h \in V_h(\Omega_h)$,}
\end{aligned}
\right.
\end{equation}
where 
\begin{equation*}
\label{eq:Acip_Special_def}
\Acip^{\rm ad}(u_h, v_h) \coloneqq \sum_{E \in \Omega_h} \Acip^{{\rm ad},E} (u_h, v_h) \, ,
\end{equation*}
and
\begin{equation*}
\label{eq:Acip_special}
\Acip^{{\rm ad},E} (u_h, v_h) \coloneqq \epsilon a_h^E(u_h , v_h) + \bskewEh(u_h , v_h) + \mathcal{N}_h(u_h, v_h) + J_h^E(u_h , v_h) \, .
\end{equation*}

The key observation is that a suitable inf-sup condition still holds true without the help of the $L^2$-type norm stemming from the reaction term. In fact, introducing the local norm

\begin{equation*} 
\|v_h\|^2_{\cip, {\rm ad}, E} 
:= 
\epsilon \, \| \nabla v_h \|^2_{0,E} + 
h \, \| \bb \cdot \PZ0P \nabla v_h \|^2_{0,E} +  
\Vert \xi (\epsilon, \bb) v_h \Vert^2_{0,\Gamma_E} +
J_h^E(v_h,v_h) \, ,
\end{equation*}
and its global counterpart
\[
\|v_h\|^2_{\cip, E}  \coloneqq \sum_{E \in \Omega_h} \|v_h\|^2_{\cip, {\rm ad}, E} \, ,
\]
similarly to Proposition~\ref{prp:inf-sup}, we have the following result.
\begin{proposition}\label{prp:inf-sup-ad}
Under assumptions \textbf{(A1)}, it holds that
\begin{equation}\label{eq:inf-sup-ad_case}
\| v_h \|_{\cip,{\rm ad}} 
\lesssim 
\sup_{z_h \in V_h(\Omega_h)} \dfrac{\Acip^{\rm ad} (v_h, z_h)}{\| z_h \|_{\cip , {\rm ad}}}
\qquad \text{for all $v_h \in V_h(\Omega_h)$,}    
\end{equation}
for a constant that does not depend on $h$ and $\epsilon$.
\end{proposition} 

\begin{proof}
The proof of \eqref{eq:inf-sup-ad_case} is analogous to the one of Proposition~\ref{prp:inf-sup}, with the simplification that in Lemma \ref{lm:adv-term} it holds $\bb_h=\bb$. The main difference stands in the treatment of $\eta_{\bb_1}$, $\eta_{\bb_2}$ and $\eta_{\bb_3}$ in $\mathbf{(T_5)}$, see \eqref{eq:b_split}. These terms are the only ones requiring the help of the $L^2$ norm in the general case (apart, of course, the reaction term itself). 
Regarding the term $\eta_{\bb_1}$, in our present case we immediately have the advective norm:
\begin{equation}\label{eq:bhinfsup1-spec}
\eta_{\bb_1} =
h \, \| \bb \cdot \nabla \P0 v_h \|^2_{0,E}
\end{equation}
Since $\bb\in\Pk_1(\Omega)$, it follows that $\bb \cdot \nabla \P0 v_h\in\Pk_k(E)$, so that we can directly bound $\eta_{\bb_2}$ using Young's inequality and Proposition \ref{prp:clmest}:
\begin{equation} \label{eq:b-bh-spec}
\begin{aligned}
 \eta_{\bb_2} & = h \bigl( \bb \cdot \nabla \P0 v_h, ( \pi - I ) (\bb \cdot \nabla  \P0 v_h) \bigr)_{0,E}\\
&\geq   - \dfrac{h}{2} \| \bb\cdot \nabla \P0 v_h \|^2_{0,E} -
\dfrac{h}{2} \| (\pi - I) (\bb \cdot \nabla \P0 v_h) \|^2_{0,E} \\
& \geq  - \dfrac{h}{2} \| \bb\cdot \nabla \P0 v_h \|^2_{0,E} - C\, J_h^{\mathcal{D}(E)}(v_w,v_h)  \, , 
\end{aligned}
\end{equation}
which is the counterpart of \eqref{eq:bhinfsup2}.
Furthermore, again since $\bb \cdot \nabla \P0 v_h\in\Pk_k(E)$, it follows that 
$$
\eta_{\bb_3} := \bigl(\bb\cdot \nabla\P0 v_h, \, (\P0 - I) w_h\bigr)_{0,E} = 0\, .
$$
\end{proof}
Once the above stability result has been established, the next error estimate can be proved using the same arguments of Proposition \ref{prp:bfb}. 
The only difference is handling the terms $\eta_{\mathcal{F}}^E$ and $\eta_{b,1}^E$ which now must be bounded using diffusion (since reaction is not available) and therefore paying a price in terms of $\varepsilon$ but with a better rate in terms of $h$. For the sake of conciseness we here omit the simple alternative derivations for such terms.
\begin{proposition}
\label{prp:bfb-2}
Under the assumptions \textbf{(A1)} and \textbf{(A2)}, let $u \in V$ be the
solution of equation \eqref{eq:problem-c_special} and $u_h \in V_h(\Omega_h)$ be the solution of equation \eqref{eq:cip-vem_special}. 
Then it holds that
\[
\begin{aligned}
\|u - u_h\|^2_{\cip,{\rm ad}} &\lesssim
\sum_{E \in \Omega_h} \Theta^E
\left(
\epsilon \, h^{2\reg} 
+   
h^{2\reg+1}
+
\dfrac{h^{2(\reg + 2)}}{\epsilon}
\right) \, ,
\end{aligned}
\]
where the constant $\Theta^E$ depends on
$\|u\|_{s+1,E}$, $\|f\|_{s+1,E}$, $\frac{\|\beta\|_{[W^{s+1}_{\infty}(E)]^2}}{\beta_E}$.
\end{proposition}

\section{Numerical Experiment}
\label{sec:num}

In this section, we investigate the actual computational behavior of the proposed method.

\paragraph{Model problem.}
We consider a family of problems in the unit square $\Omega=(0,\,1)^2$. We choose the boundary conditions and the source term (which turns out to depend on $\epsilon$, $\sigma$ and $\bb$) in such a way that the analytical solution is always the function
$$
u(x,\,y) := \sin(\pi\,x)\sin(\pi\,y)\,.
$$
Different choices of the parameters $\sigma$, $\epsilon$ and of the advective term $\bb(x,y)$ will be selected. Since the pointwise values of the numerical solution $u_h$ are unknown, the following error quantities will be considered:

\begin{itemize}
 \item \textbf{$H^1-$seminorm error}
 $$
 e_{H^1} := \sqrt{\sum_{E\in\mathcal{T}_h}\left\|\nabla(u-\Pi_k^\nabla u_h)\right\|^2_{0,E}}\,;
 $$
 
 \item \textbf{$L^2-$norm error}
 $$
 e_{L^2} := \sqrt{\sum_{E\in\mathcal{T}_h}\left\|(u-\P0 u_h)\right\|^2_{0,E}}\,.
 $$
\end{itemize}
We will consider two different families of mesh:
\begin{itemize}
 \item \texttt{quad:} a mesh composed by highly distorted quadrilaterals obtained perturbing a mesh composed of structured squares;
 \item \texttt{voro:} a centroidal Voronoi tessellation of the unit square.
\end{itemize}
These two families are represented in Figure~\ref{fig:meshes}.

\begin{figure}[!htb]
\begin{center}
\begin{tabular}{ccc}
\texttt{quad} & &\texttt{poly} \\
\includegraphics[width=0.35\textwidth]{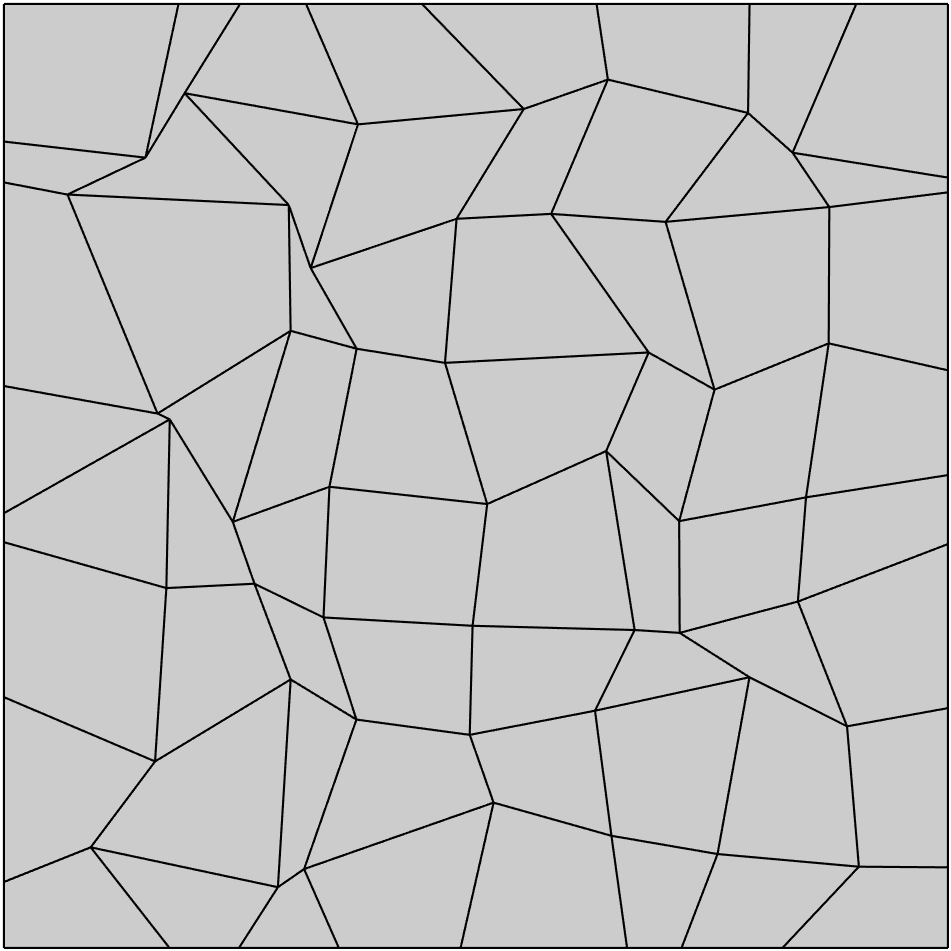} &\phantom{mm}&
\includegraphics[width=0.35\textwidth]{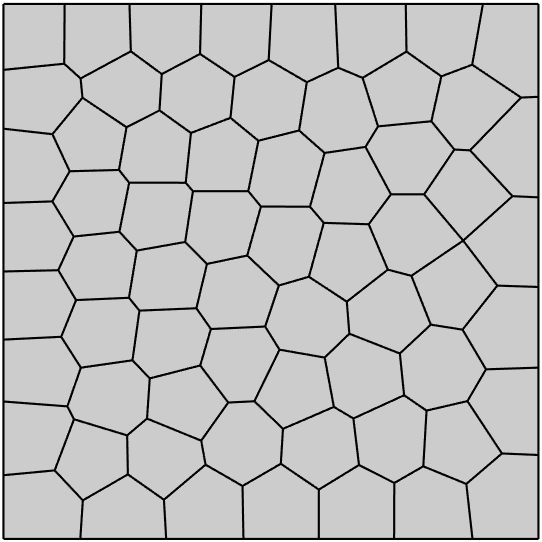} \\
\end{tabular}
\end{center}
\caption{Example of meshes used for the present test case.}
\label{fig:meshes}
\end{figure}

\paragraph{Effects of the CIP stabilization.} The first aspect we investigate is the benefits of inserting the CIP term in the variational formulation of the problem. We thus consider an advection-dominated regime and choose the parameters $\epsilon = 1e-9$, $\sigma = 0$, along with a constant advection term 
$$
\bb(x,\,y) := \left[\begin{array}{r}
1\\
0.5
\end{array}\right]\, .
$$
We consider a centroidal Voronoi tesselation of the domain $\Omega$ into 256 polygons.
The degree of the method is set to $k=1$.

In Figure~\ref{fig:benefits} we observe that by inserting the bilinear form $J_h(\cdot,\cdot)$ in the variational formulation, we are able to accurately approximate the analytic solution $u(x,y)$ of the model problem. 
If we omit the CIP term, we obtain (as expected) a definitely unsatisfactory numerical solution, which exhibits nonphysical oscillations all over the computational domain. We also remark that these instabilities reach peaks of the order of $10^2$, despite for the analytic solution we have $\| u \|_{L^\infty(\Omega)}=1$.

\begin{figure}[!htb]
\begin{center}
\begin{tabular}{ccc}
\texttt{No CIP} & &\texttt{CIP} \\
\includegraphics[width=0.43\textwidth]{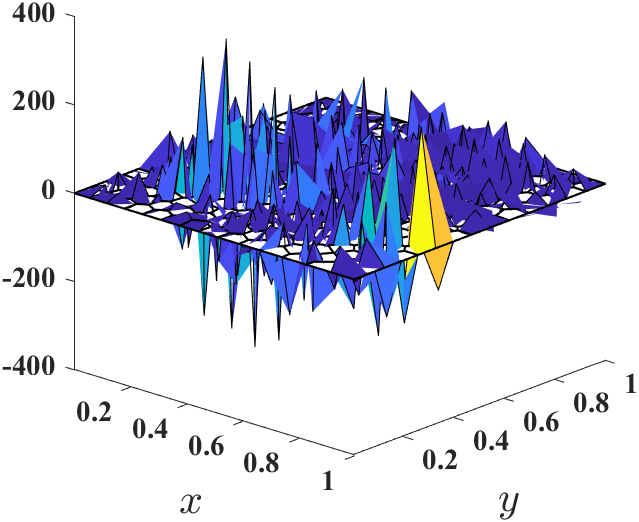} &\phantom{mm}&
\includegraphics[width=0.43\textwidth]{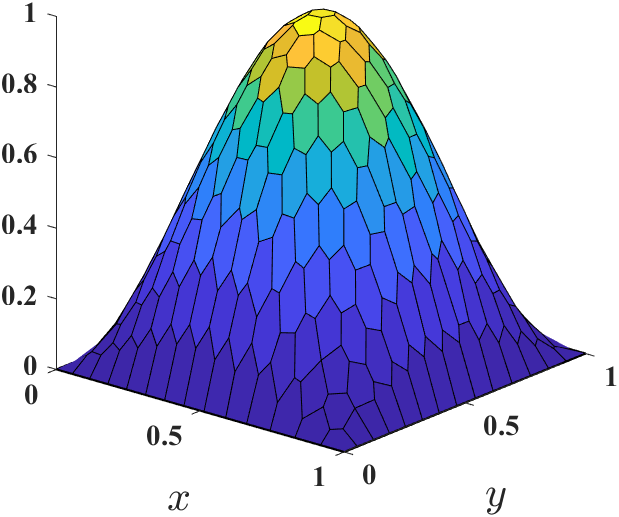} \\
\end{tabular}
\end{center}
\caption{Effects of the CIP stabilaizing term.}
\label{fig:benefits}
\end{figure}

\paragraph{Convergence analysis} We now investigate the convergence of the numerical method by means of the above-introduced norms, and choosing different consistency order, i.e. $k=1,2,3$.
The convective term is
$$
\bb(x,\,y) := \left[\begin{array}{r}
1\\
0.5
\end{array}\right]\, .
$$
We consider a diffusion-dominated case ($\epsilon = 1$), and an advection-dominated one ($\epsilon = 1e-9$). 
Thus, we are in the framework of Section \ref{ss:p1beta}. Accordingly, we neglet the reaction term (hence, $\sigma = 0$) and the theoretical error bound of Proposition \ref{prp:bfb-2} holds.
We compare the method with and without the jump term $J_h(\cdot,\cdot)$. The results are obtained using the Voronoi mesh family.

\begin{figure}[!htb]
\begin{center}
$\epsilon = 1$
\begin{tabular}{ccc}
\includegraphics[width=0.43\textwidth]{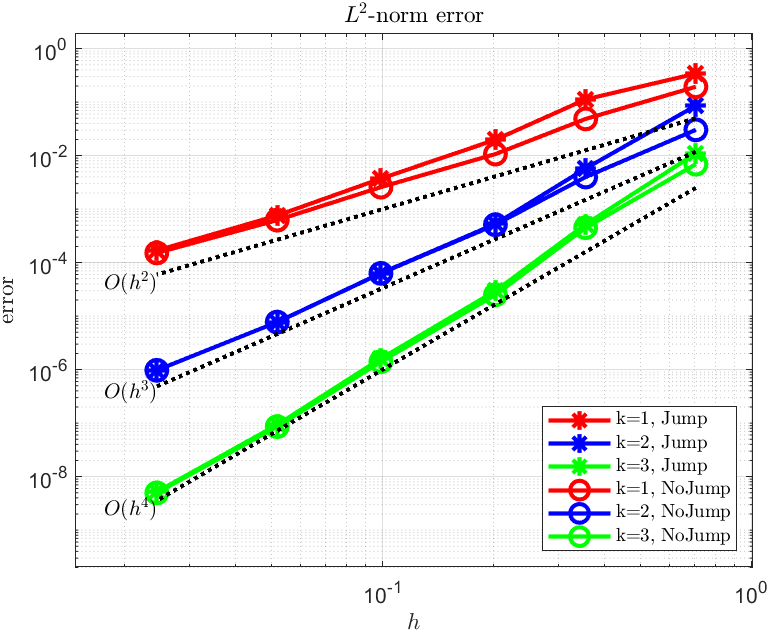} &\phantom{mm}&
\includegraphics[width=0.43\textwidth]{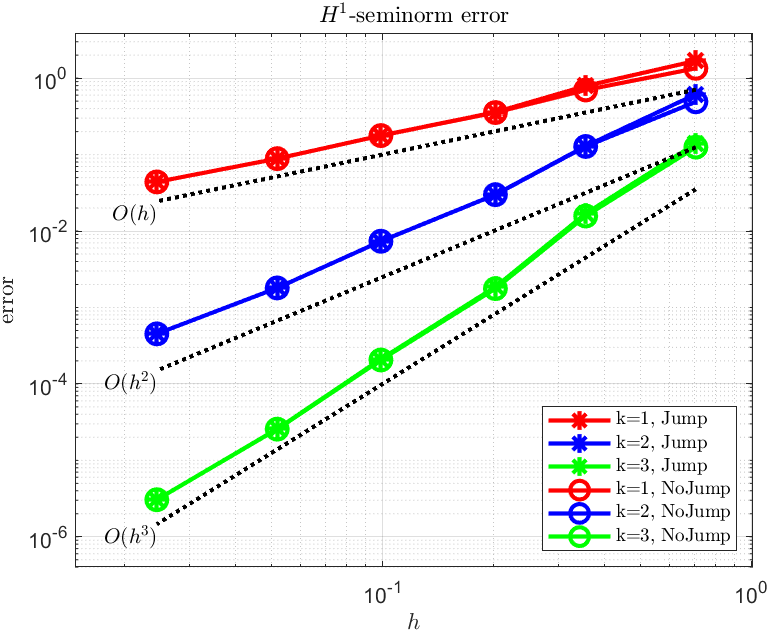} 
\end{tabular}
$\epsilon = 10^{-9}$
\begin{tabular}{ccc}
\includegraphics[width=0.43\textwidth]{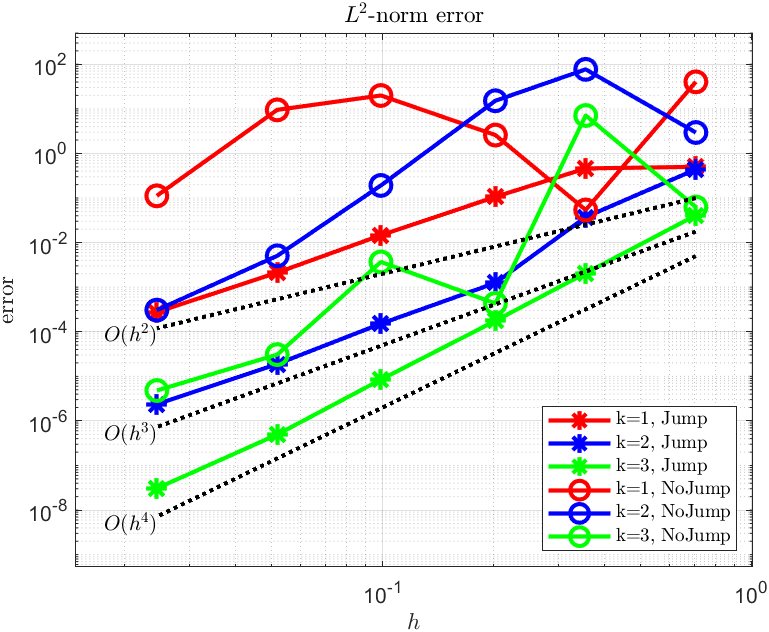} &\phantom{mm}&
\includegraphics[width=0.43\textwidth]{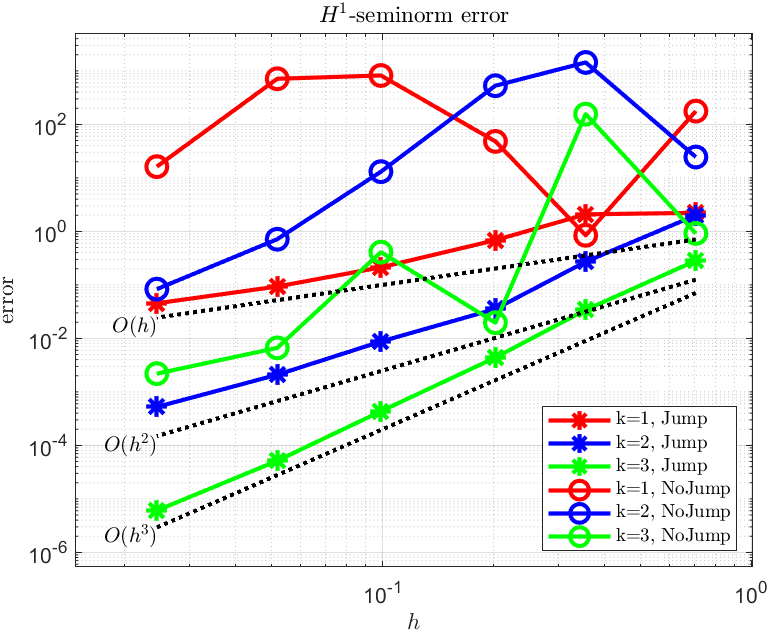} \\
\end{tabular}
\end{center}
\caption{Convergences in the case $\epsilon = 1$ and $\epsilon = 10^{-9}$.}
\label{fig:convergence}
\end{figure}

In Figure~\ref{fig:convergence}, we observe that in the case $\epsilon = 1$ the two methods behave in the same way. Instead, in the advection-dominated regime we observe that the optimal convergences are attained when inserting the stabilising jump term; without it, as expected, the method displays unsatisfactory results, especially for the low-order case.

\paragraph{Effect of the reaction term.} We now consider an advection-dominated problem with a  variable advection term not in $\Pk_1(\Omega)$.  
In particular, we select
$$
\bb(x,\,y) := \left[\begin{array}{r}
-2\,\pi\,\sin(\pi\,(x+2\,y))\\
\pi\,\sin(\pi\,(x+2\,y))
\end{array}\right]\,.
$$
We recall that for this case we are able to prove robust error bounds only with the aid of the reaction term, see Proposition \ref{prp:bfb}.
The diffusive coefficient is set to $\epsilon = 1e-9$. We consider two different families of mesh. The first one is made by the usual Voronoi polygons and the second one is composed of distorted squares.  
We select two different values for the reaction term: $\sigma = 1$ and $\sigma = 0$. 
\begin{figure}[!htb]
\begin{center}
$\sigma = 1$
\begin{tabular}{ccc}
\includegraphics[width=0.43\textwidth]{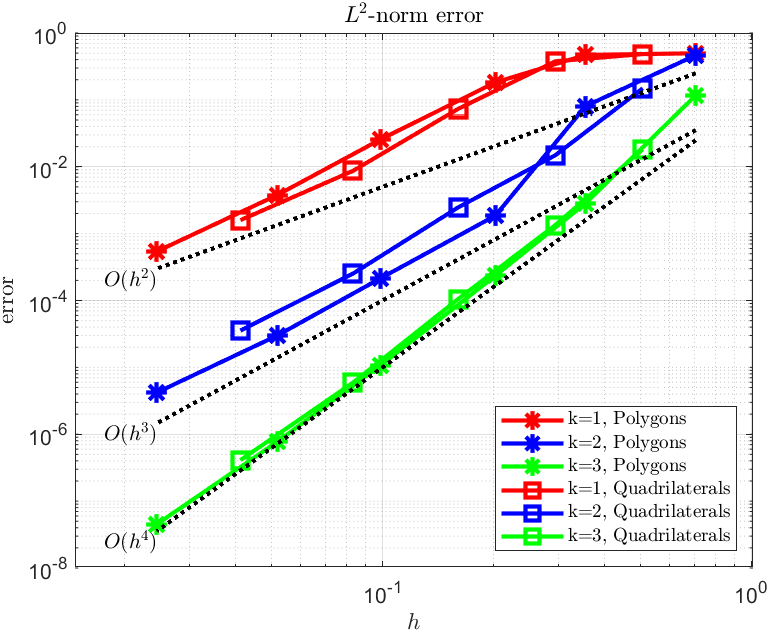} &\phantom{mm}&
\includegraphics[width=0.43\textwidth]{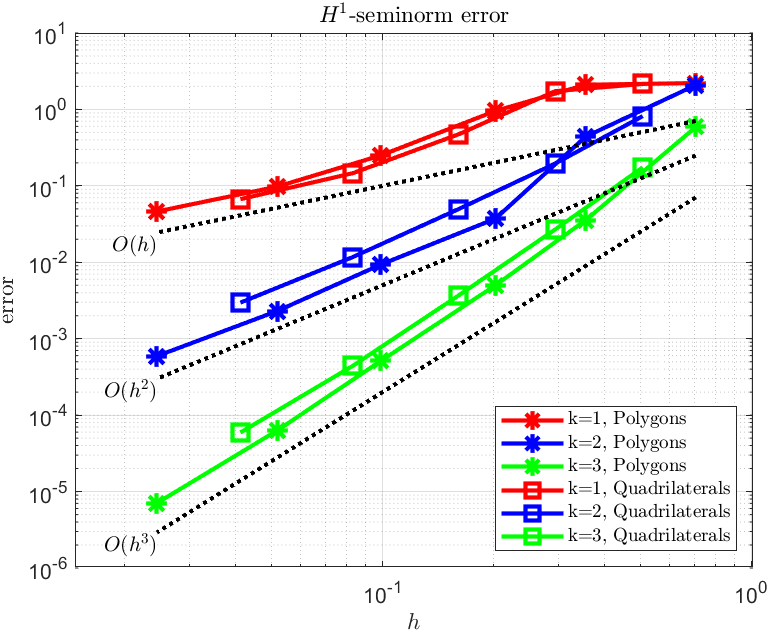} 
\end{tabular}
$\sigma = 0$
\begin{tabular}{ccc}
\includegraphics[width=0.43\textwidth]{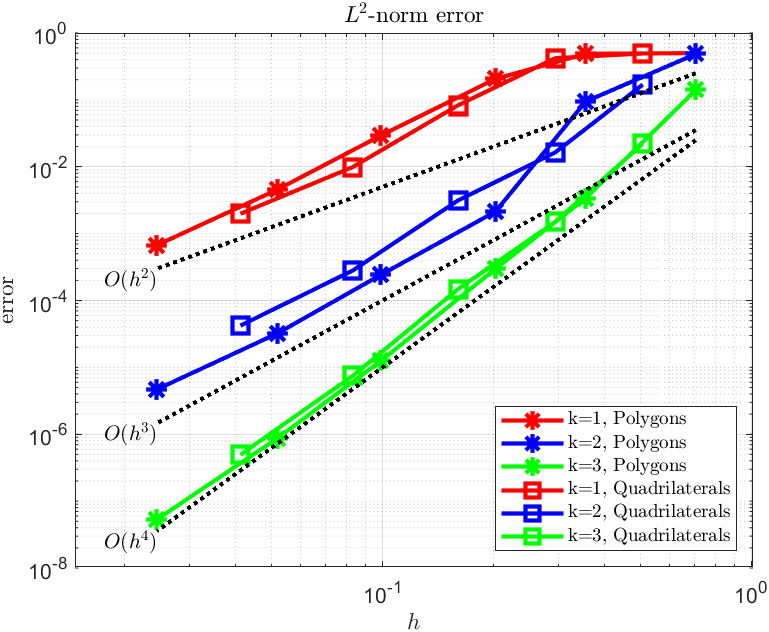} &\phantom{mm}&
\includegraphics[width=0.43\textwidth]{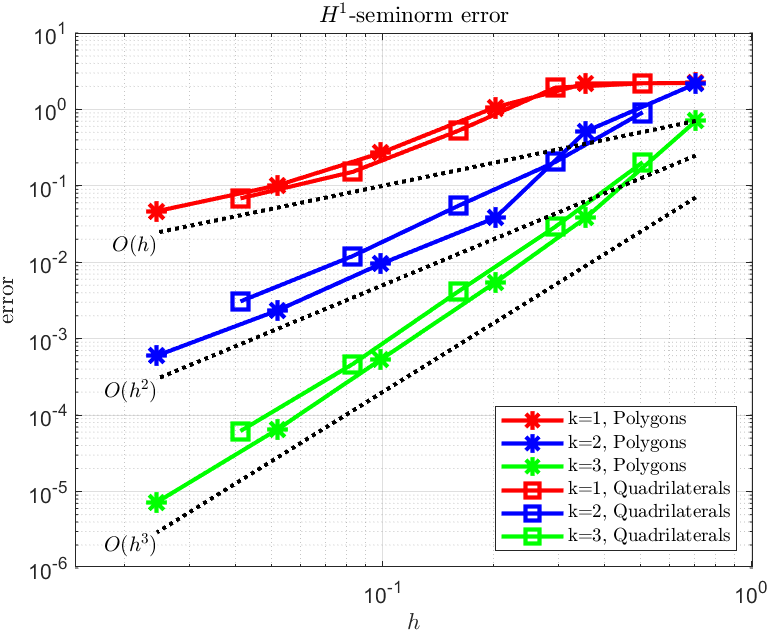} \\
\end{tabular}
\end{center}
\caption{Convergences in the case $\sigma = 0$ and $\sigma = 1$.}
\label{fig:simga}
\end{figure}

Figure~\ref{fig:simga} shows that there is no significant difference between the cases $\sigma=1$ and $\sigma =0$. As already mentioned, for this latter case Proposition \ref{prp:bfb} does not apply, and no satisfactory theoretical analysis is available, yet. However, the numerical outcomes seems to suggest that it could be possible to drop the reaction term even if the advection term is not globally linear. We note also that we achieve a good convergence also in the case that the mesh is composed of unstructured quadrilaterals.

\addcontentsline{toc}{section}{\refname}
\bibliographystyle{plain}
\bibliography{biblio}

\begin{thebibliography}{10}

\bibitem{sema-simai}
{\em The Virtual Element Method and its applications}, volume~31.
\newblock SEMA-SIMAI springer series, 2021.
\newblock P.F. Antonietti, L. Beirao da Veiga, G. Manzini Editors.

\bibitem{andersen17}
Odd Andersen, Halvor~M Nilsen, and Xavier Raynaud.
\newblock Virtual element method for geomechanical simulations of reservoir
  models.
\newblock {\em Computational Geosciences}, 21:877--893, 2017.

\bibitem{volley}
L.~Beir\~{a}o~da Veiga, F.~Brezzi, A.~Cangiani, G.~Manzini, L.~D. Marini, and
  A.~Russo.
\newblock Basic principles of {V}irtual {E}lement {M}ethods.
\newblock {\em Math. Models Methods Appl. Sci.}, 23(1):199--214, 2013.

\bibitem{hitchhikers}
L.~Beir\~{a}o~da Veiga, F.~Brezzi, L.~D. Marini, and A.~Russo.
\newblock The hitchhiker's guide to the virtual element method.
\newblock {\em Mathematical Models and Methods in Applied Sciences},
  24(08):1541--1573, 2014.

\bibitem{BDLV:2021}
L.~Beir\~{a}o~da Veiga, F.~Dassi, C.~Lovadina, and G~Vacca.
\newblock {SUPG}-stabilized virtual elements for diffusion-convection problems:
  A robustness analysis.
\newblock {\em ESAIM Math. Model. Numer. Anal.}, 55(5):2233--2258, 2021.

\bibitem{BLR:2017}
L.~Beir\~{a}o~da Veiga, C.~Lovadina, and A.~Russo.
\newblock Stability analysis for the virtual element method.
\newblock {\em Math. Mod.and Meth. in Appl. Sci.}, 27(13):2557--2594, 2017.

\bibitem{ACTA-VEM}
L.~Beirao~da Veiga, F.~Brezzi, L.D. Marini, and A.~Russo.
\newblock The virtual element method.
\newblock {\em ACTA Numerica}, 32:123--202, 2023.

\bibitem{berrone:2016}
M.~F. Benedetto, S.~Berrone, A.~Borio, S.~Pieraccini, and S.~Scialò.
\newblock Order preserving {SUPG} stabilization for the virtual element
  formulation of advection-diffusion problems.
\newblock {\em Comput. Methods Appl. Mech. Engrg.}, 293:18--40, 2016.

\bibitem{gen_Benedetto2014a}
M.~F. Benedetto, S.~Berrone, S.~Pieraccini, and S.~Scial{\`o}.
\newblock The virtual element method for discrete fracture network simulations.
\newblock {\em Comput. Methods Appl. Mech. Engrg.}, 280:135--156, 2014.

\bibitem{brenner-guan-sung:2017}
S.~C. Brenner, Q.~Guan, and L.Y. Sung.
\newblock Some estimates for virtual element methods.
\newblock {\em Comput. Methods Appl. Math.}, 17(4):553--574, 2017.

\bibitem{brenner-scott:book}
S.~C. Brenner and L.~R. Scott.
\newblock {\em The Mathematical Theory of Finite Element Methods}, volume~15 of
  {\em Texts in Applied Mathematics}.
\newblock Springer, New York, third edition, 2008.

\bibitem{brenner-sung:2018}
S.~C. Brenner and L.Y. Sung.
\newblock Virtual element methods on meshes with small edges or faces.
\newblock {\em Math. Models Methods Appl. Sci.}, 28(7):1291--1336, 2018.

\bibitem{BrezziDG}
Franco Brezzi and Endre Süli.
\newblock Discontinuous galerkin methods for first-order hyperbolic problems.
\newblock {\em Mathematical Models and Methods in Applied Sciences}, 14, 02
  2004.

\bibitem{brooks1982}
Alexander~N. Brooks and Thomas~J.R. Hughes.
\newblock Streamline upwind/petrov-galerkin formulations for convection
  dominated flows with particular emphasis on the incompressible navier-stokes
  equations.
\newblock {\em Computer Methods in Applied Mechanics and Engineering},
  32(1):199--259, 1982.

\bibitem{burman:2007}
E.~Burman and A.~Ern.
\newblock Continuous interior penalty hp-finite element methods for advection
  and advection-diffusion equations.
\newblock {\em Mathematics of Computation}, 76(259):1119--1140, 2007.

\bibitem{burman:2004}
E.~Burman and P.~Hansbo.
\newblock Edge stabilization for galerkin approximations of
  convection–diffusion–reaction problems.
\newblock {\em Computer Methods in Applied Mechanics and Engineering},
  193(15):1437--1453, 2004.

\bibitem{cangiani:2017}
A.~Cangiani, E.H. Georgoulis, T.~Pryer, and O.J. Sutton.
\newblock A posteriori error estimates for the virtual element method.
\newblock {\em Numer. Math.}, 137(4):857--893, 2017.

\bibitem{HHO-book-2}
Matteo Cicuttin, Alexandre Ern, and Nicolas Pignet.
\newblock {\em Hybrid high-order methods: a primer with applications to solid
  mechanics}.
\newblock Springer, 2021.

\bibitem{HHO-book-1}
Daniele~Antonio Di~Pietro and J{\'e}r{\^o}me Droniou.
\newblock The hybrid high-order method for polytopal meshes.
\newblock {\em Number 19 in Modeling, Simulation and Application}, 2020.

\bibitem{DE-book}
Daniele~Antonio Di~Pietro and Alexandre Ern.
\newblock {\em Mathematical aspects of discontinuous Galerkin methods},
  volume~69.
\newblock Springer Science \& Business Media, 2011.

\bibitem{douglas}
Jim Douglas and Todd Dupont.
\newblock Interior penalty procedures for elliptic and parabolic galerkin
  methods.
\newblock In R.~Glowinski and J.~L. Lions, editors, {\em Computing Methods in
  Applied Sciences}, pages 207--216, Berlin, Heidelberg, 1976. Springer Berlin
  Heidelberg.

\bibitem{LPS0}
Petr Knobloch and Gert Lube.
\newblock Local projection stabilization for advection--diffusion--reaction
  problems: One-level vs. two-level approach.
\newblock {\em Applied numerical mathematics}, 59(12):2891--2907, 2009.

\bibitem{li2021}
Yang Li and Minfu Feng.
\newblock A local projection stabilization virtual element method for
  convection-diffusion-reaction equation.
\newblock {\em Applied Mathematics and Computation}, 411:126536, 2021.

\bibitem{matthies2007unified}
Gunar Matthies, Piotr Skrzypacz, and Lutz Tobiska.
\newblock A unified convergence analysis for local projection stabilisations
  applied to the oseen problem.
\newblock {\em ESAIM: Mathematical Modelling and Numerical Analysis},
  41(4):713--742, 2007.

\bibitem{R-book}
B{\'e}atrice Rivi{\`e}re.
\newblock {\em Discontinuous Galerkin methods for solving elliptic and
  parabolic equations: theory and implementation}.
\newblock SIAM, 2008.

\end{thebibliography}
\end{document}